\documentclass[10pt]{amsart}
\usepackage[margin=1.5in]{geometry}
\usepackage{xcolor}
\usepackage{xcolor}
\usepackage{commath}
\newtheorem{theorem}{Theorem}[section]
\newtheorem{lemma}[theorem]{Lemma}
\newtheorem{proposition}{Proposition}[section]

\theoremstyle{definition}
\newtheorem{definition}[theorem]{Definition}

\usepackage{cite}

\theoremstyle{remark}
\newtheorem{remark}[theorem]{Remark}

\numberwithin{equation}{section}

\begin{document}
\title[The Shakhov model near a global Maxwellian]{The Shakhov model near a global Maxwellian}

\author{Gi-Chan Bae}
\address{Research institute of Mathematics, Seoul National University, Seoul 08826, Republic of Korea }
\email{gcbae02@snu.ac.kr}

\author{Seok-Bae Yun}
\address{Department of mathematics, Sungkyunkwan University, Suwon 16419, Republic of Korea }
\email{sbyun01@skku.edu}

%    General info
%\subjclass[2010]{76P05,46N55,82C40,35F25   }

%\date{January 1, 2001 and, in revised form, June 22, 2001.}


\begin{abstract}
	Shakhov model is a relaxation approximation of the Boltzmann equation proposed to overcome the deficiency of the original BGK model, namely, the incorrect production of the Prandtl number.
	In this paper, we address the existence and the asymptotic stability of the Shakhov model when the initial data is a small perturbation of global equilibrium. 
	We derive a dichotomy in the coercive estimate of the linearized relaxation operator between zero and non-zero Prandtl number, and observe that the linearized relaxation operator is more degenerate in the former case. To remove such degeneracy and recover the full coercivity, we consider a micro-macro system that involves an additional non-conservative quantity related to the heat flux.
\end{abstract}
\maketitle
%\tableofcontents
\section{Introduction}
\subsection{The Shakhov model}
The fundamental model describing the dynamics of rarefied gases at the mesoscopic level is the Boltzmann equation. 
But the complicated structure and the high dimensionality have long hindered the practical application of the Boltzmann equation.
In this regard, the model equation introduced in \cite{BGK,MR0062041},  has been popularly used to study various flow problems in place of the Boltzmann equation. 
However, it was soon revealed that this model, which goes by the BGK model, cannot achieve the correct Navier-Stokes limit in
that the Prandtl number computed in the hydrodynamic limit is incorrect.

There have been two major remedies to overcome this drawback.
The first such effort goes back to Holway \cite{MR0204022}, who extended the local Maxwellian into an ellipsoidal Gaussian 
to obtain an additional degree of freedom in computing the transport coefficients.
%This model did not attract much interest until the H-theorem was proved later in \cite{MR1803856},which greatly popularized the model.  
On the other hand, Shakhov \cite{shakhov1968generalization} suggested a way to get the correct Prandtl number by multiplying the Maxwellian with
an extra term that adjusts the heat flux while leaving the collision invariant untouched. The price to pay for this adjustment is that the H-theorem holds only when the flow remains close to the fluid regime. Even with such apparent defect of the model, the Shakhov model has been
widely used in various fields of rarefied gas dynamics, since it reproduces satisfactory qualitative features of the Boltzmann dynamics in many important flow problems \cite{cai2012nr,MR3389025,MR2327060,MR3943439,MR3925938,zheng2005ellipsoidal}. 
Especially, due to the non-triviality of the heat flux compared to that of the original BGK, it is reported that the Shakhov
model works better for non-isothermal flows \cite{MR3925938,MR3389025}.  To the best of the authors' knowledge, however, the existence of this important kinetic model has never been studied in the mathematical literature, which is the main motivation of the current work.

More precisely, we consider in this paper the following initial value problem of the Shakhov model:
\begin{align}\label{Shakhov}
\begin{split}
\partial_tF+v\cdot \nabla_xF &= \frac{1}{\tau}(\mathcal{S}_{Pr}(F)-F), \cr
F(x,v,0) &= F_0(x,v).
\end{split}
\end{align}
The unknown $F:\mathbb{T}^3\times\mathbb{R}^3\times\mathbb{R}_+\rightarrow\mathbb{R}$ is called the velocity distribution function where $F(x,v,t)$ is the number density of molecules in the phase space on the phase point $(x,v)$ at time $t$. We define the macroscopic density $\rho$, bulk velocity $U$, temperature $T$, stress tensor $\Theta$, and the heat flux $q$ as follows:
\begin{align}\label{macro quantity}
\begin{split}
\rho(x,t)&=\int_{\mathbb{R}^3}F(x,v,t)dv,\cr
\rho(x,t) U(x,t)&=\int_{\mathbb{R}^3}F(x,v,t)vdv,\cr
3\rho(x,t) T(x,t)&=\int_{\mathbb{R}^3}F(x,v,t)|v-U(x,t)|^2dv,\cr
\rho\Theta_{ij}(x,t)&=\int_{\mathbb{R}^3}F(x,v,t)(v_i-U_i(x,t))(v_j-U_j(x,t)) dv,\cr
q(x,t) &=\int_{\mathbb{R}^3}F(x,v,t)(v-U(x,t))|v-U(x,t)|^2dv.
\end{split}
\end{align}
The Shakhov operator is defined as
\begin{align}\label{Shakhov operator}
	\begin{split}
		\mathcal{S}_{Pr}(F)(x,v,t)=\mathcal{M}(F)\left[1+\frac{1-Pr}{5}\frac{q(x,t)\cdot (v-U(x,t))}{\rho(x,t) T(x,t)^2}\left(\frac{|v-U(x,t)|^2}{2T(x,t)}-\frac{5}{2}\right)\right],
	\end{split}
\end{align}
where $Pr$ is the Prandtl number, $\mathcal{M}(F)$ is the standard local Maxwellian:
\begin{align*}
	\mathcal{M}(F)(x,v,t)= \frac{\rho(x,t)}{\sqrt{2\pi T(x,t)}^3}\exp\left(-\frac{|v-U(x,t)|^2}{2T(x,t)}\right).
\end{align*}
Although the stress tensor $\Theta$ does not explicitly appears in the definition (\ref{Shakhov operator}), except through the relation:
\begin{align}\label{Ttheta}
3T=\sum_{1\leq i \leq 3 }\Theta_{ii},
\end{align}
we listed it in (\ref{macro quantity}) since it will be crucially used in the analysis later.
The relaxation time  $\tau$ takes the following form \cite{MR2327060,cai2012nr,pantazis2010heat,MR2779616}:
\begin{align}\label{tau}
\frac{1}{\tau} = \frac{1}{\tau_0}\rho^{\eta} T^w,
\end{align}
for some positive constant $\tau_0$ and $\eta\geq 0$, $w\in\mathbb{R}$.

The Shakhov operator satisfies the following identities by construction (See Appendix.):
\begin{align*}
\int_{\mathbb{R}^3} \mathcal{S}_{Pr}(F)(x,v,t)\left(\begin{array}{c}1\cr v \cr |v|^2\end{array}\right)dv = \int_{\mathbb{R}^3} F(x,v,t)\left(\begin{array}{c}1\cr v \cr |v|^2\end{array}\right)dv,
\end{align*}
%\begin{align*}
%\int_{\mathbb{R}^3} \left(\begin{array}{c}1\cr v+U \cr |v+U|^2\end{array}\right)\frac{\rho}{\sqrt{2\pi T}^3}\exp\left(-\frac{|v|^2}{2T}\right)\left[1+\frac{1-Pr}{5}\frac{q\cdot  v}{\rho T^2}\left(\frac{|v|^2}{2T}-\frac{5}{2}\right)\right]dv=0
%\end{align*}
which implies the conservation laws of the total mass, momentum, and energy:
\begin{align}\label{conserv}
\begin{split}
&\frac{d}{dt}\int_{\mathbb{T}^3\times\mathbb{R}^3} F(x,v,t)dvdx = 0, \cr
&\frac{d}{dt}\int_{\mathbb{T}^3\times\mathbb{R}^3} F(x,v,t)vdvdx= 0, \cr
&\frac{d}{dt}\int_{\mathbb{T}^3\times\mathbb{R}^3} F(x,v,t)|v|^2dvdx= 0.
\end{split}
\end{align}
The $H$-theorem is proved only when the $F$ is sufficiently close to the equilibrium in \cite{shakhov1968generalization} (See Appendix.): 
\begin{align}\label{Hthm}
\frac{d}{dt}\int_{\mathbb{T}^3\times\mathbb{R}^3} F\ln F dvdx\leq 0.
\end{align}

We note that the Shakhov relaxation operator satisfies
\begin{align}\label{rq}
\int_{\mathbb{R}^3}(\mathcal{S}_{Pr}(F)-F)(v_i-U_i)|v-U|^2dv = -Pr q_i(x,t).
\end{align}

This additional cancellation property  explains why the Shakhov model has a bigger degeneracy in this case in the vanishing Prantl number regime (See Proposition \ref{dichotomy}). \\
Finally, we mention that the Shakhov model is a generalization of the BGK model in the sense that it reduces to the original BGK model when $Pr=1$.

\subsection{Main results}

In this paper, we consider the existence and asymptotic behavior of (\ref{Shakhov}) when the initial data is close enough to the normalized global Maxwellian:
\begin{align}\label{global}
m(v)= \frac{1}{\sqrt{(2\pi)^3}}e^{-\frac{|v|^2}{2}}.
\end{align}
%which has the same macroscopic fields of initial function  
%\begin{align*}
%\int_{\mathbb{T}^3\times\mathbb{R}^3} F_0(x,v)\left(\begin{array}{c}1\cr v \cr |v|^2\end{array}\right)dvdx =\left(\begin{array}{c}1\cr 0 \cr 3\end{array}\right)
%\end{align*}
We define the perturbation $f$ by $F = m + \sqrt{m}f$, and derive following equation for $f$ from \eqref{Shakhov}:
\begin{align}\label{perturb1}
\begin{split}
\partial_tf+v\cdot \nabla_xf &= \frac{1}{\tau_0}L_{Pr}f+\Gamma(f), \cr
f(x,v,0)&=f_0(x,v).
\end{split}
\end{align}
Here, $f_0(x,v)= (F_0(x,v)-m) /\sqrt{m}$,  $L_{Pr}$, and $\Gamma(f)$ are the linear part and the non-linear part of the linearized relaxation operator (See Section 2 for precise definitions). 

We define the energy functional:
\begin{align}\label{energy}
\mathcal{E}(f)(t) = \frac{1}{2}\sum_{|\alpha|+|\beta|\leq N }\|\partial^{\alpha}_{\beta}f(t)\|_{L^2_{x,v}}^2+\sum_{|\alpha|+|\beta|\leq N}\int_{0}^{t}\|\partial^{\alpha}_{\beta}f(s)\|^2_{L^2_{x,v}}ds,
\end{align}
where $\|\cdot\|_{L^2_{x,v}}$ is the standard $L^2$ norm:
\begin{align*} \|f\|^2_{L^2_{x,v}}=\int_{\mathbb{T}^3\times\mathbb{R}^3}|f(x,v)|^2dvdx,
\end{align*}
and the multi-indices notation was employed for the differential operator: 
\[\partial^{\alpha}_{\beta} = \partial_t^{\alpha_0}\partial_{x_1}^{\alpha_1}\partial_{x_2}^{\alpha_2}\partial_{x_3}^{\alpha_3}\partial_{v_1}^{\beta_1}\partial_{v_2}^{\beta_2}\partial_{v_3}^{\beta_3},\]
with $\alpha =(\alpha_0,\alpha_1,\alpha_2,\alpha_3)$,  $\beta =(\beta_1,\beta_2,\beta_3)$.

We now state the main result of this paper.

\begin{theorem}\label{theorem} Let $N\geq 3$, and fix $Pr\geq0$. Assume  the initial data satisfies $F_0(x,v)=m+\sqrt{m}f_0(x,v)\geq 0 $, and shares the same total mass, momentum and energy with the global equilibrium $m(v)$:
\begin{align*}
\int_{\mathbb{T}^3\times\mathbb{R}^3} F_0(x,v)\left(\begin{array}{c}1\cr v \cr |v|^2\end{array}\right)dvdx = \int_{\mathbb{T}^3\times\mathbb{R}^3} m(v)\left(\begin{array}{c}1\cr v \cr |v|^2\end{array}\right)dvdx.
\end{align*}
In the case of $Pr=0$, we assume further that the total third moment of the initial data vanishes:
\[\int_{\mathbb{T}^3\times\mathbb{R}^3} F_0(x,v)v|v|^2dvdx= 0.\] 
Then, there exists $M$ such that if $\mathcal{E}(f_0) \leq M$, then there exists the unique global-in-time classical solution $f$ to \eqref{perturb1} satisfying 
\begin{enumerate}
\item The distribution function $F$ and the Shakhov operator are non-negative for $t \geq 0$: 
\begin{align*}
F(x,v,t)=m+\sqrt{m}f(x,v,t) \geq 0, \quad \mathcal{S}_{Pr}(F)(x,v,t)\geq 0,
\end{align*}
and satisfies the conservation laws \eqref{conservf}.
\item The energy functional is uniformly bounded: 
\begin{align*}
\sup_{t \geq 0}\mathcal{E}(f)(t) \leq C \mathcal{E}(f_0).
\end{align*}
\item The perturbation $f$ decays exponentially fast: 
\begin{align*}
\sum_{|\alpha|+|\beta|\leq N }\|\partial^{\alpha}_{\beta}f(t)\|_{L^2_{x,v}} \leq Ce^{-\delta t},
\end{align*}
for some positive constant $C>0$ and $\delta>0$.
\item Let $f$ and $\bar{f}$ be solutions corresponding to initial data $f_0$ and $\bar{f}_0$, respectively satisfying $\mathcal{E}(f_0)\leq M$ and $\mathcal{E}(\bar{f}_0)\leq M$. Then there exists positive constant $C$  satisfying the following $L^2$ stability estimate:
\begin{align*}
\|f(t)-\bar{f}(t)\|_{L^2_{x,v}} \leq C\|f_0-\bar{f}_0\|_{L^2_{x,v}}.
\end{align*}
\end{enumerate}
\end{theorem}
\begin{remark}
The study of the vanishing Prandtl number is not only mathematically interesting but also physically relevant even since it can be an approximate model for fluids with a very small Prandtl number. For example, in some thermoacoustic engines, a gas with a low Prandtl number plays a critical role because of its high heat diffusivity. For example, the authors in \cite{mixpr,mixprmin} observed that some mixtures of light and heavy noble gases produce low Prandtl numbers. %And the minimum Prandtl number 0.12 is obtained by a mixture of He and Xe in room temperature \cite{mixprmin}.
Besides, zero-Prandtl-number limit is considered in \cite{prconvection} to study the convection with extremely small Prandtl number such as %the Earth's liquid-core ($Pr \approx 10^{-1}$) and 
the convection zone of the sun ($Pr \approx 10^{-8}$).
\end{remark}
\begin{remark}
In the case of $Pr>0$, the instant energy functional without the production term is sufficient to close the energy estimate as in \cite{MR4096124}. But when $Pr=0$, the production term must be incorporated into the energy functional. (See details in Proposition \ref{abcd} step 3.) 
\end{remark}

\subsection{Novelties $\&$ difficulties} To close the energy estimate to extend the local solution into the global one, it is important to identify the dissipative nature of the linear term $L_{Pr}$. In this regard, we observe the following dichotomy  in the degeneracy of the dissipative estimate: When $Pr>0$, we have
\begin{align}\label{P_c coer}
\langle L_{Pr}f,f \rangle_{L^2_v}\leq -\min\{Pr,1\} \|(I-P_c)f \|_{L^2_v}^2,
\end{align}
where $P_c$ is a projection operator on the linear space spanned by $\{\sqrt{m},v\sqrt{m},|v|^2\sqrt{m}\}$.
On the other hand,  the dissipation becomes more degenerate in the case $Pr=0$:
\begin{align}\label{P_{nc} coer}
\langle L_{0}f,f \rangle_{L^2_{x,v}} = -\|(I-P_c-P_{nc})f\|_{L^2_{x,v}}^2,
\end{align}
where $P_{nc}f$ is a projection operator on the linear space spanned by non-conservative basis $\{v|v|^2\sqrt{m}\}$, so that $P_c+P_{nc}$ constitutes
a projection operator on a wider space spanned by 8 bases: 
$\{\sqrt{m},v\sqrt{m},|v|^2\sqrt{m},v|v|^2\sqrt{m} \}$.

This additional degeneracy in the vanishing Prandtl number regime is due to the following cancellation property unobserved in the original BGK model or the Boltzmann equation:
\begin{align*}
\int_{\mathbb{R}^3}(\mathcal{S}_{0}(F)-F)(v_i-U_i)|v-U|^2dv = 0,
\end{align*}
which is obtained by putting $Pr=0$ in \eqref{rq}.
The larger kernel \eqref{P_{nc} coer} indicates that the degeneracy is stronger in the vanishing Prantl number regime since the dissipativity of $L_{Pr}$ stops operating on the null space.
In the case $Pr>0$, the full dissipativity of $L_{Pr}$ can be recovered by the standard argument \cite{MR1908664,MR2095473,MR3357626}: the derivatives of the kernels are estimated using the micro-macro equations that govern the evolution of the degenerate part, and the lowest order estimates are derived by combining the derivative estimates and the Poincar\'{e} inequality with the vanishing moments of the perturbation up to the second order. 
On the other hand, when the Prandtl number vanishes, a novel difficulty unobserved in the previous literature arises: The micro-macro system now involves a non-conservative quantity, namely the heat flux part. This leads to a more complicated micro-macro system, and more seriously, the lowest order estimate of these new terms can not be treated in a similar manner as in the previous case, since the third-order moment of the perturbation related to the heat flux does not vanish. To overcome this, we assume that the total third moment of the initial data vanishes:
\[\int_{\mathbb{T}^3\times\mathbb{R}^3} F_0(x,v)v_i|v|^2dvdx= 0.\]
and use the evolution law for the third moment:

\begin{align*}
\frac{d}{dt}\int_{\mathbb{T}^3\times\mathbb{R}^3} fv_i(|v|^2-5) \sqrt{m}dvdx
= \frac{1}{\tau} \bigg( \int_{\mathbb{T}^3}2U_i\rho T - \sum_{1\leq j \leq 3}2\rho U_j\Theta_{ij}dx\bigg),
\end{align*} 
to derive the following modified dissipation estimate: 
\begin{align*}
\sum_{|\alpha|\leq N}\langle L_{Pr}\partial^{\alpha}f, \partial^{\alpha}f\rangle_{L^2_{x,v}} \leq -\delta \sum_{|\alpha|\leq N}\|\partial^{\alpha}f\|_{L^2_{x,v}}^2+C\mathcal{E}^{2}(t),
\end{align*}
for some positive constant $\delta>0$, which enables one to construct the global-in-time classical solution.
\iffalse
The Shakhov model has some drawbacks that the distribution function $F$ and the Shakhov operator can take a negative value.
However, if the initial data $F_0$ starts sufficiently close to the equilibrium, then we can have the following estimate: 
\begin{align*}
\mathcal{S}_{Pr}(F)\geq \mathcal{M}(F)\left(1-C\frac{1-Pr}{5}\frac{|q|}{\rho T^2}\right),
\end{align*}
which guarantees the positivity of the Shakhov operator $\mathcal{S}_{Pr}(F)$ in the whole time. To identify it, when we take an iteration scheme, we should prove the positivity of $\mathcal{S}_{Pr}(F^n)$ in each iteration step. Then the positivity of $\mathcal{S}_{Pr}(F^n)$ lead to the positivity of the next iteration solution: $F^{n+1} \geq 0$.

We mention that, the reason that we use energy functional with time integrated production term $\int^t_0\|\partial^{\alpha}_{\beta}f\|dt$ is that in the reconstruction of the above full coercivity of
the linearized relaxation operator, uniform control of a non-conservative quantity  - the heat flux - is needed, which is realized only when the energy functional has production term. This is in start constrast with most of other collisional or relaxational kinetic equations where the energy functional consisting only of instant energy term $\|\partial^{\alpha}_{\beta}f\|$ are sufficient to close the energy estimate.
energy term is
\newline
\fi

\subsection{Brief history}
We briefly overview mathematical results on various BGK models. \newline
\noindent(1) Original BGK model:
The first existence result goes back to \cite{MR1023307} where Perthame established the existence of global weak solutions of the BGK model. 
The uniqueness was guaranteed in a more stringent weighted $L^{\infty}$ space in \cite{MR1245074} (See \cite{MR2370243} for $L^p$ extension). The strong convergence to the Maxwellian is obtained by Desvillettes in \cite{MR1031086}. For the stationary case, Ukai constructs the existence theorem with a large boundary data in a $1$-dimensional bounded interval.
The global-in-time classical solution of the BGK  model near equilibrium can be found in \cite{MR1969901,MR2779616}. For the particle system immersed in a fluid described by the coupled equation of the Navier-Stokes equation and
the BGK model, we refer to \cite{CY-NHM,CY-NS,CY-NS-Strong}.
For the convergence analysis of numerical schemes of the BGK model, see \cite{Issau,RSY}. 
\newline
\noindent(2) ES-BGK model: The revival of interest in this model was brought due to the proof of the H-theorem provided in \cite{MR1803856}. A systematic derivation of this model was suggested in \cite{BS}.
The existence of a classical solution in the weighted $L^{\infty}$ space was obtained in \cite{MR3397316}.
In \cite{MR3357626}, the asymptotic stability of global Maxwellians is considered. 
The entropy-entropy production estimate is derived in \cite{MR3541542,Kim-Lee-Yun}. 
Convergence analysis of a fully discretized scheme for ES-BGK model was made in \cite{RY}. 
For the mathematical results on the polyatomic version of the ES-BGK model, we refer to \cite{PY,PY2,PY3,MR3912760}.
\newline
\noindent (3) Shakhov model: 
The mathematical research of the Shakhov model is in the initial state.
In \cite{MR3662471}, the author considered the existence of the stationary $1$-dimensional steady state. Latyshev and Yushkanov analytically solved a stationary boundary value problem of the Shakhov type equation in \cite{MR2028190}.
For the numeric scheme of the Shakhov model, we refer to \cite{cai2012nr,pantazis2010heat,MR2327060,MR3943439}. 

A brief review of the literature that compares the ES-BGK model and the Shakhov model is in order.
\cite{graur2009comparison} illustrates some numerical examples which indicate that the ES-BGK model works better than the Shakhov model for heat transfer problems. On the other hand, it is reported in \cite{MR3389025} that the Shakhov model produces more accurate results under tough conditions such as the shock structure or some particular boundary conditions in the transition regime \cite{MR3389025}. Besides, the Shakhov model can capture the velocity slip and the temperature jump near the wall more accurately, and shows good accuracy in predicting the non-equilibrium flow in transition regime \cite{MR3925938}. See \cite{zheng2005ellipsoidal} for an organized comparison between these models.

The general mathematical and physical review of the Boltzmann and the BGK equation can be found in \cite{MR1803856,MR1215852,MR1313028,MR1307620,MR0258399,MR1014927,MR1379589,MR0156656,MR363332,MR1942465}.
\newline

\subsection{Notations:}
The following notations, conventions, and definitions will be fixed throughout the paper.
\begin{itemize}
	\item The constant $C$ in an estimate denotes a generically defined constant. 
	\item $(x_1,\cdots,x_n)$ is understood as an $n$-dimensional column vector.
	\item $\mathbb{I}_n$ is the $n$-tuple of $1$: $(1,\cdots,1)\in\mathbb{R}^n$.
	\item  $0^n$ stands for the $n$-dimension zero vector. For example $(1,0^3)=(1,0,0,0)$. 
	\item We use the standard $L^2_{v}$ and $L^2_{x,v}$ inner product on $\mathbb{R}^3_v$ and $\mathbb{T}^3_x\times\mathbb{R}^3_v$, respectively.
	\begin{align*}
	\langle f,g \rangle_{L^2_v}=\int_{\mathbb{R}^3}f(v)g(v)dv, \quad \langle f,g \rangle_{L^2_{x,v}}=\int_{\mathbb{T}^3\times\mathbb{R}^3}f(x,v)g(x,v)dvdx.
	\end{align*}
	\item We use the standard $L^2_v$ norm and $L^2_{x,v}$ norm on $\mathbb{R}^3_v$ and $\mathbb{T}^3_x\times\mathbb{R}^3_v$, respectively.
	\begin{align*}
	\|f\|_{L^2_v}=\left(\int_{\mathbb{R}^3}|f(v)|^2dv\right)^{\frac{1}{2}}, \quad \|f\|_{L^2_{x,v}}=\left(\int_{\mathbb{T}^3\times\mathbb{R}^3}|f(x,v)|^2dvdx\right)^{\frac{1}{2}}.
	\end{align*}
	\item We use the multi-indices notations $\alpha =(\alpha_0,\alpha_1,\alpha_2,\alpha_3)$,  $\beta =(\beta_1,\beta_2,\beta_3)$, and differential operator: 
	\[\partial^{\alpha}_{\beta} = \partial_t^{\alpha_0}\partial_{x_1}^{\alpha_1}\partial_{x_2}^{\alpha_2}\partial_{x_3}^{\alpha_3}\partial_{v_1}^{\beta_1}\partial_{v_2}^{\beta_2}\partial_{v_3}^{\beta_3}.\]
	\item Throughout this paper, we fix $\gamma$ to denote pure temporal derivative:
	\[
	\gamma=(\gamma_0,0,0,0),
	\]
	so that
	\[
	\partial^{\gamma}f=\partial^{\gamma_0}_tf.
	\]
\end{itemize}

This paper is organized as follows: In Section 2, we linearize the Shakhov model near a global Maxwellian. %and find that the linear term is $8$-dimensional macroscopic projection of $f$. 
Section 3 is devoted to proving some estimates of the macroscopic fields and the non-linear term to construct the local-in-time solution. In Section 4 and Section 5, we establish the coercivity estimate for $Pr>0$ and $Pr=0$, respectively. %When $Pr=0$, we derive the modified coercivity estimate for the additional linear terms. 
From the coercivity estimate, we establish the existence of the global-in-time classical solution in the last section. %and the asymptotic behavior of the solution.

\section{Linearization}
\subsection{Linearization of the Shakhov operator}
In this section, we linearize the Shakhov operator $\mathcal{S}_{Pr}(F)$ near the global Maxwellian. We start with the definition of the Shakhov projection operator.
\begin{definition}\label{P_{Pr}f} We define the  $8$-dimensional macroscopic projection operator as follows:
\begin{align*}
P_{Pr}f&=P_cf+ (1-Pr) P_{nc}f,
\end{align*}
where  $P_c$ and $P_{nc}$ are projections on the conservative space and non-conservative space respectively:
\begin{align*}
P_cf&=\left(\int_{\mathbb{R}^3} f\sqrt{m}dv\right)\sqrt{m}+\left(\int_{\mathbb{R}^3} fv\sqrt{m}dv\right)\cdot v\sqrt{m}+\left(\int_{\mathbb{R}^3} f\frac{|v|^2-3}{\sqrt{6}}\sqrt{m}dv\right)\frac{|v|^2-3}{\sqrt{6}}\sqrt{m}, \cr
P_{nc}f&=\left(\int_{\mathbb{R}^3} f\frac{v(|v|^2-5)}{\sqrt{10}} \sqrt{m}dv\right)\cdot\frac{v(|v|^2-5)}{\sqrt{10}}\sqrt{m}.
\end{align*}
\end{definition}

%For a technical reason, we define some macroscopic quantities $G$ and $H$, and they will constitute an orthonormal basis.
\begin{definition}\label{GH} We define $G_{ij}$ and $H_i$ $(1\leq i,j\leq 3)$ by
\begin{align}\label{G}
G_{ij}=\begin{cases}\frac{1}{2}\left(\rho \Theta_{ii}+\rho U_i^2-\rho\right), ~ &\textit{if} \quad i=j,	\cr
\rho\Theta_{ij}+\rho U_iU_j, ~ &\textit{if} \quad i\neq j,\end{cases}
\end{align}
for $1\leq i,j \leq 3$, and
\begin{align}\label{H}
H_i = \frac{1}{\sqrt{10}}\bigg(q_i +\sum_{1\leq j \leq 3}2\rho U_j\Theta_{ij}+ \rho U_i|U|^2 + \rho U_i (\Theta_{11}+\Theta_{22}+\Theta_{33}) -5\rho U_i\bigg),
\end{align}
for $i=1,\cdots,3$, where $\Theta$ is the stress tensor defined in (\ref{macro quantity}). Since $G$(or $\Theta$) is symmetric $3 \times 3$ matrix, we view it as a component of $\mathbb{R}^6$:
\begin{align}\label{order}
	G= \left\{G_{11},G_{22},G_{33},G_{12},G_{23},G_{31}\right\},
\end{align}
for simplicity.
\end{definition}

\begin{proposition}\label{linear} Let $F=m+\sqrt{m}f$. Then the Shakhov operator is linearized into the following form: 
\begin{align*}
\mathcal{S}_{Pr}(F) = m+P_{Pr}f\sqrt{m} + \sum_{1 \leq i,j \leq 13} \int_0^1 \left\{D^2_{(\rho_{\theta},\rho_{\theta} U_{\theta},G_{\theta},H_{\theta})}\mathcal{S}_{Pr}(\theta)\right\}_{ij}(1-\theta)d\theta \langle f, e_i \rangle_{L^2_v}\langle f, e_j \rangle_{L^2_v},
\end{align*}
where 
\begin{align*}
\mathcal{S}_{Pr}(\theta)=\frac{\rho_{\theta}}{\sqrt{2\pi T_{\theta}}^3}\exp\left(-\frac{|v-U_{\theta}|^2}{2T_{\theta}}\right)\left[1+\frac{1-Pr}{5}\frac{q_{\theta}\cdot (v-U_{\theta})}{\rho_{\theta} T_{\theta}^2}\left(\frac{|v-U_{\theta}|^2}{2T_{\theta}}-\frac{5}{2}\right)\right],
\end{align*}
and the transitional  macroscopic fields $\rho_{\theta}$, $U_{\theta}$, $\Theta_{\theta}$ and $q_{\theta}$ are given by 
\begin{align}\label{transmacro}
\rho_{\theta} &= \theta \rho +(1-\theta), \quad
\rho_{\theta}U_{\theta} = \theta \rho U, \quad
G_{\theta}= \theta G, \quad
H_{\theta}=\theta H.
\end{align}
The $13$-basis $\{e_i\}_{1\leq i \leq 13}$ denote 
\begin{align*}
	e_1&=\sqrt{m}, \qquad e_{i+1}=v_i\sqrt{m}, \qquad e_{i+4}=\frac{v_i^2-1}{2}\sqrt{m},  \cr
	e_{8}&=v_1v_2\sqrt{m}, \qquad e_{9}=v_2v_3\sqrt{m} \qquad e_{10}=v_1v_3\sqrt{m},
\end{align*}
and
\begin{align*}
	e_{i+10} = \frac{v_i|v|^2\sqrt{m}-5v_i\sqrt{m}}{\sqrt{10}},
\end{align*}
for $i=1,2,3$.
\end{proposition}
\begin{proof}
%Note that the Shakhov operator $\mathcal{S}_{Pr}(F)$ depends on $8$-dimensional macroscopic fields $(\rho,U,T,q)$. However, we proceed the linearization process through the $13$-dimensional macroscopic fields $\rho$, $U$, $G$ and $H$, since $H$ is related to $G$. To expand $\mathcal{S}_{Pr}(\theta)$, 
We apply Taylor's theorem:
\begin{align}\label{taylor}
\mathcal{S}_{Pr}(1)=\mathcal{S}_{Pr}(0)+\mathcal{S}_{Pr}'(0)+\int_0^1 \mathcal{S}_{Pr}''(\theta)(1-\theta) d\theta.
\end{align}
%We substitute $\theta=0$ on \eqref{transmacro} to have
%\begin{align*}	(\rho_{\theta}, \rho_{\theta} U_{\theta}, G_{\theta}, H_{\theta})|_{\theta=0} = (1,0^3,0^6,0^3).\end{align*}
%Substituting it in the relations in \eqref{G} and \eqref{H} gives \begin{align}\label{rutq0}
%(\rho_{\theta}, U_{\theta}, \Theta_{\theta} ,q_{\theta})|_{\theta=0} =(1,0^3,1^3,0^3,0^3). \end{align}
We can easily see that
\begin{align*}
\mathcal{S}_{Pr}(0)= m, \qquad \textit{and} \qquad \mathcal{S}_{Pr}(1)&=\mathcal{S}_{Pr}(F).
\end{align*}
\noindent$\bullet$ Computation of $\mathcal{S}_{Pr}'(0)$: An explicit computation with a change of variable 
\begin{align}\label{changeofvar}
	(\rho,U,\Theta,q) \rightarrow (\rho,\rho U,G,H),
\end{align}
gives
\begin{align}\label{S'}
	\mathcal{S}_{Pr}'(0)&=\left(\frac{d (\rho_{\theta}, \rho_{\theta} U_{\theta}, G_{\theta}, H_{\theta})}{d \theta}\right)^{T} \left( \frac{\partial(\rho_{\theta},\rho_{\theta} U_{\theta},G_{\theta},H_{\theta})} {\partial(\rho_{\theta},U_{\theta},\Theta_{\theta},q_{\theta})} \right)^{-1}\nabla_{(\rho_{\theta},U_{\theta},\Theta_{\theta},q_{\theta})}\mathcal{S}_{Pr}(\theta)\bigg|_{\theta=0}.
\end{align}
To proceed further, we need the following two auxiliary lemmas.

\begin{lemma}\label{Jaco} 
	Let $J$ denote the Jacobian matrix:
	\[
	J \equiv \frac{\partial(\rho,\rho U,G,H)} {\partial(\rho,U,\Theta,q)}.
	\]
	Then we have\newline
\noindent(1) T $J$ is given by
\begin{align*}
\left[ {\begin{array}{ccccccccccccc}
1 & 0 & 0 & 0 & 0 & 0 & 0 & 0 & 0 & 0 & 0 & 0 & 0  \\
U_1 &&& & 0 & 0 & 0 & 0 & 0 & 0 & 0 & 0 & 0\\
U_2 && \rho I_3&&0 & 0 & 0 & 0 & 0 & 0& 0 & 0 & 0\\
U_3 & &&& 0 & 0 & 0 & 0 & 0 & 0 & 0 & 0 & 0  \\
(\Theta_{11}+U_1^2-1)/2 & \rho U_1 & 0 & 0 & &&& 0 & 0 & 0 & 0 & 0 & 0 \\
(\Theta_{22}+U_2^2-1)/2 & 0 & \rho U_2 & 0 & &\frac{\rho}{2}I_3& & 0 & 0 & 0 & 0 & 0 & 0\\
(\Theta_{33}+U_3^2-1)/2 & 0 & 0 & \rho U_3 & &&& 0 & 0 & 0 & 0 & 0& 0 \\
\Theta_{12}+U_1U_2 & \rho U_2 & \rho U_1 & 0 & 0 & 0 & 0 &&& & 0 & 0 & 0 \\
\Theta_{23}+U_2U_3 & 0 & \rho U_3 & \rho U_2 & 0 & 0 & 0 & &\rho I_3& & 0 & 0 & 0\\
\Theta_{31}+U_3U_1 & \rho U_3 & 0 & \rho U_1 & 0 & 0 & 0 & && & 0 & 0& 0 \\
 &  &  &  &  &  &  & \frac{2\rho U_2}{\sqrt{10}} & 0 & \frac{2\rho U_3}{\sqrt{10}} & && \\
A &  & B &  &  & C &  & \frac{2\rho U_1}{\sqrt{10}} & \frac{2\rho U_3}{\sqrt{10}} & 0 & &\frac{1}{\sqrt{10}}I_3&\\
 &  &  &  &  &  &  & 0 & \frac{2\rho U_2}{\sqrt{10}} & \frac{2\rho U_1}{\sqrt{10}}& &&
\end{array} } \right],
\end{align*}
where
\begin{align*}
A_i&=\frac{1}{\sqrt{10}}\left(\sum_{1\leq j \leq 3}2 U_j\Theta_{ij}+  U_i|U|^2 +  U_i (\Theta_{11}+\Theta_{22}+\Theta_{33}) -5\rho U_i\right),\cr
B_{ij}&=\frac{1}{\sqrt{10}}\left(2\rho\Theta_{ij}+2\rho U_iU_j +(\rho(\Theta_{11}+\Theta_{22}+\Theta_{33}) + \rho |U|^2-5\rho)\delta_{ij}\right),
\end{align*}
and
\begin{align*}
C_{ij}=\frac{1}{\sqrt{10}}\left(2\rho U_i\delta_{ij}+\rho U_j\right),
\end{align*}
for $i,j=1,2,3$. \newline
\noindent(2) The inverse of $J$ reads
\begin{align*}
\left[ {\begin{array}{ccccccccccccc}
1 & 0 & 0 & 0 & 0 & 0 & 0 & 0 & 0 & 0 & 0 & 0 & 0  \\
-\frac{U_1}{\rho} &&& & 0 & 0 & 0 & 0 & 0 & 0 & 0 & 0 & 0\\
-\frac{U_2}{\rho} && \frac{1}{\rho} I_3&&0 & 0 & 0 & 0 & 0 & 0& 0 & 0 & 0\\
-\frac{U_3}{\rho} & &&& 0 & 0 & 0 & 0 & 0 & 0 & 0 & 0 & 0  \\
\frac{-\Theta_{11}+U_1^2+1}{\rho} & -\frac{2U_1}{\rho} & 0 & 0 & &&& 0 & 0 & 0 & 0 & 0 & 0 \\
\frac{-\Theta_{22}+U_2^2+1}{\rho} & 0 & -\frac{2U_2}{\rho} & 0 & &\frac{2}{\rho}I_3& & 0 & 0 & 0 & 0 & 0 & 0\\
\frac{-\Theta_{33}+U_3^2+1}{\rho} & 0 & 0 & -\frac{2U_3}{\rho} & &&& 0 & 0 & 0 & 0 & 0& 0 \\
\frac{-\Theta_{12}+U_1U_2}{\rho} & -\frac{U_2}{\rho} & -\frac{U_1}{\rho} & 0 & 0 & 0 & 0 &&& & 0 & 0 & 0 \\
\frac{-\Theta_{23}+U_2U_3}{\rho} & 0 & -\frac{U_3}{\rho} & -\frac{U_2}{\rho} & 0 & 0 & 0 & &\frac{1}{\rho} I_3& & 0 & 0 & 0\\
\frac{-\Theta_{31}+U_3U_1}{\rho} & -\frac{U_3}{\rho} & 0 & -\frac{U_1}{\rho} & 0 & 0 & 0 & && & 0 & 0& 0 \\
 &  &  &  &  &  &  & -2U_2 & 0 & -2U_3 & && \\
A' &  & B' &  &  & C' &  & -2U_1 & -2U_3 & 0 & &\sqrt{10}I_3&\\
 &  &  &  &  &  &  & 0 & -2U_2 & -2U_1& &&
\end{array} } \right],
\end{align*}
where
\begin{align*}
A_i'&=2 \sum_{i\neq j}\left(U_j\Theta_{ij}-U_j^2U_i\right) - A_i\sqrt{10} +\frac{\sqrt{10}}{\rho}\sum_{1\leq j \leq 3}\left(U_jB_{ij}+C_{ij}(\Theta_{jj}-U_j^2-1) \right), \cr
B'_{ij}&=\frac{1}{5\rho}\big(10\rho U_i U_j + (|U|^2-2U_i^2)\delta_{ij} -5B_{ij}\sqrt{10}+10\sqrt{10}U_jC_{ij}\big),
\end{align*}
and
\begin{align*}
C'_{ij}=-\frac{2\sqrt{10}}{\rho}C_{ij},
\end{align*}
for $i,j=1,2,3$.
\end{lemma}
\begin{proof}
We omit it since it is straightforward and tedious.
\end{proof}
%Now we compute the macroscopic derivatives of the Shakhov operator $\nabla_{(\rho_{\theta},U_{\theta},\Theta_{\theta},q_{\theta})}\mathcal{S}_{Pr}(\theta)$.
\begin{lemma}\label{S diff} We have
\begin{align*}
(1) ~ \frac{\partial \mathcal{S}_{Pr}(	\theta)}{\partial \rho_{\theta}}\bigg|_{\theta=0} &= m, \qquad \hspace{1.9cm} 
(2) ~ \frac{\partial \mathcal{S}_{Pr}(	\theta)}{\partial U_{\theta i}} \bigg|_{\theta=0} = v_im, \cr
(3) ~ \frac{\partial \mathcal{S}_{Pr}(	\theta)}{\partial \Theta_{\theta ii}}\bigg|_{\theta=0} &= \frac{|v|^2-3}{6}m, \qquad \qquad 
(4) ~ \frac{\partial \mathcal{S}_{Pr}(	\theta)}{\partial \Theta_{\theta ij}}\bigg|_{\theta=0} =0 \quad (\textit{for} \quad i\neq j), \cr
(5) ~ \frac{\partial \mathcal{S}_{Pr}(	\theta)}{\partial q_{\theta i}}\bigg|_{\theta=0} &= \left[\frac{1-Pr}{5}v_i\left(\frac{|v|^2}{2}-\frac{5}{2}\right)\right]m,
\end{align*}
for $i,j=1,2,3.$
\end{lemma}
\begin{proof}
All these identities follow from substituting 
\begin{align*}
(\rho_{\theta}, U_{\theta}, \Theta_{\theta} ,T_{\theta}, q_{\theta})|_{\theta=0} =(1,0^3,1^3,0^3,1,0^3).
\end{align*}
into the following identities:
\begin{align*}
\frac{\partial \mathcal{S}_{Pr}(F)}{\partial \rho}=\frac{1}{\rho}\mathcal{M}(F).
\end{align*}

\begin{align*}
\frac{\partial \mathcal{S}_{Pr}(F)}{\partial U_i}&=\frac{v_i-U_i}{T}\left[1+\frac{1-Pr}{5}\frac{q\cdot (v-U)}{\rho T^2}\left(\frac{|v-U|^2}{2T}-\frac{5}{2}\right)\right]\mathcal{M}(F) \cr
&-\left[\frac{1-Pr}{5}\frac{q_i}{\rho T^2}\left(\frac{|v-U|^2}{2T}-\frac{5}{2}\right)+\frac{1-Pr}{5}\frac{q\cdot (v-U)}{\rho T^2}\left(\frac{v_i-U_i}{T}\right)\right]\mathcal{M}(F),
\end{align*}

\begin{align*}
\frac{\partial \mathcal{S}_{Pr}(F)}{\partial \Theta_{ii}} &= \frac{\partial T}{\partial \Theta_{ii}}\frac{\partial \mathcal{S}_{Pr}(F)}{\partial T} \cr
&=\frac{1}{3}\left(-\frac{3}{2T}+\frac{|v-U|^2}{2T^2}\right)\left[1+\frac{1-Pr}{5}\frac{q\cdot (v-U)}{\rho T^2}\left(\frac{|v-U|^2}{2T}-\frac{5}{2}\right)\right]\mathcal{M}(F) \cr
&-\left[\frac{1-Pr}{5}\frac{q\cdot (v-U)}{\rho T^2}\left(\frac{|v-U|^2}{2T^2}-\frac{5}{3T}\right)\right]\mathcal{M}(F),
\end{align*}
\begin{align*}
\frac{\partial \mathcal{S}_{Pr}(F)}{\partial \Theta_{ij}}&=0,
\end{align*}
\begin{align*}
	\frac{\partial \mathcal{S}_{Pr}(F)}{\partial q_i}&=\left[\frac{1-Pr}{5}\frac{v_i-U_i}{\rho T^2}\left(\frac{|v-U|^2}{2T}-\frac{5}{2}\right)\right]\mathcal{M}(F).
\end{align*}
\end{proof}
Now we turn back to  \eqref{S'}, and evaluate each term in \ref{S'} at $\theta=0$. 
From the definition of the transitional macroscopic fields \eqref{transmacro}, we compute
\begin{align}\label{ddt}
\frac{d (\rho_{\theta}, \rho_{\theta} U_{\theta}, G_{\theta}, H_{\theta})}{d \theta}&=(\rho-1, \rho U,G, H). %\cr
%&= \left(\int_{\mathbb{R}^3}f\sqrt{m} dv ,\int_{\mathbb{R}^3}fv\sqrt{m} dv, \int_{\mathbb{R}^3}f\frac{v_i^2-1}{2}\sqrt{m} dv , \int_{\mathbb{R}^3}fv_iv_j\sqrt{m} dv , \int_{\mathbb{R}^3}fv\frac{|v|^2-5}{\sqrt{10}}\sqrt{m} dv\right) \cr
%&= \int_{\mathbb{R}^3}\left(1,v,\frac{v_1^2-1}{2},\frac{v_2^2-1}{2},\frac{v_3^2-1}{2},v_1^2 ,v_2^2,v_3^2,v_1v_2,v_2v_3,v_1v_3,v_1\frac{|v|^2-5}{\sqrt{10}},v_2\frac{|v|^2-5}{\sqrt{10}},v_3\frac{|v|^2-5}{\sqrt{10}} \right)f\sqrt{m} dv
\end{align}
We express each term as a moment of $f$.
First, substituting $F=m+\sqrt{m}f$ into $\eqref{macro quantity}_1$ and $\eqref{macro quantity}_2$, we easily get 
\begin{align}\label{rh}
\rho-1=\int_{\mathbb{R}^3}\sqrt{m}fdv, \qquad \rho U=\int_{\mathbb{R}^3}\sqrt{m}fvdv.
\end{align}
For $G$ and $H$, we rewrite  $\eqref{macro quantity}_4$ and $\eqref{macro quantity}_5$ as
\begin{align}\label{viv2}
\begin{split}
\int_{\mathbb{R}^3}F(x,v,t)v_iv_jdv&= \rho \Theta_{ij}+\rho U_iU_j, \cr
\int_{\mathbb{R}^3}F(x,v,t)v_i|v|^2dv&=q_i +\sum_{1\leq j \leq 3}2\rho U_j\Theta_{ij}+ \rho U_i|U|^2 + U_i \rho  \sum_{1\leq i \leq 3 }\Theta_{ii},
\end{split}
\end{align}
so that
\begin{align*}
\begin{split}
\int_{\mathbb{R}^3}\sqrt{m}fv_iv_jdv&= \begin{cases}\rho \Theta_{ii}+\rho U_i^2 - 1, &\textit{if} \quad i=j, \cr \rho \Theta_{ij}+\rho U_iU_j,~ &\textit{if} \quad i\neq j,  \end{cases}
\end{split}
\end{align*}
and
\begin{align*}
\begin{split}
\int_{\mathbb{R}^3}\sqrt{m}fv_i|v|^2dv&=q_i +\sum_{1\leq j \leq 3}2\rho U_j\Theta_{ij}+ \rho U_i|U|^2 + U_i \rho  \sum_{1\leq i \leq 3 }\Theta_{ii}.
\end{split}
\end{align*}
Recalling the definition of $G$ and $H$ in \eqref{G} and \eqref{H}, these identities yield
\begin{align}\label{newG}
G_{ij}=\begin{cases}\frac{1}{2}\int_{\mathbb{R}^3}(v_i^2-1)f\sqrt{m} dv, ~ &\textit{if} \quad i=j,	\cr
\int_{\mathbb{R}^3} v_iv_j f\sqrt{m} dv , ~ &\textit{if} \quad i\neq j,\end{cases}
\end{align}
and
\begin{align}\label{newH}
H_i = \int_{\mathbb{R}^3}v_i\frac{|v|^2-5}{\sqrt{10}}f\sqrt{m} dv.
\end{align}
Replacing entries in R.H.S of (\ref{ddt}) with \eqref{rh}, \eqref{newG}, \eqref{newH}, we derive the follwing expression:
\begin{align}\label{rugh}
\begin{split}
\frac{d (\rho_{\theta}, \rho_{\theta} U_{\theta}, G_{\theta}, H_{\theta})}{d \theta} 
=\left(  \langle f, e_1  \rangle_{L^2_v}, \cdots, \langle f, e_{13} \rangle_{L^2_v}  \right).
%&= \int_{\mathbb{R}^3}\bigg(1,v_1,v_2,v_3,\frac{v_1^2-1}{2},\frac{v_2^2-1}{2},\frac{v_3^2-1}{2},\cr
%&\quad v_1v_2,v_2v_3,v_1v_3,v_1\frac{|v|^2-5}{\sqrt{10}},v_2\frac{|v|^2-5}{\sqrt{10}},v_3\frac{|v|^2-5}{\sqrt{10}} \bigg)f\sqrt{m} dv.
\end{split}
\end{align}
The Jacobian term follows directly from Lemma \ref{Jaco} (2): 
\begin{align}\label{diag0}
\begin{split}
	\left( \frac{\partial(\rho_{\theta},\rho_{\theta} U_{\theta},G_{\theta},H_{\theta})} {\partial(\rho_{\theta},U_{\theta},\Theta_{\theta},q_{\theta})} \right)^{-1}\bigg|_{\theta=0} %= \cr 
%\left[ {\begin{array}{ccccccccccccc}
%			1 & 0 & 0 & 0 & 0 & 0 & 0 & 0 & 0 & 0 & 0 & 0 & 0  \\
%			0 &&& & 0 & 0 & 0 & 0 & 0 & 0 & 0 & 0 & 0\\
%			0 &&  I_3&&0 & 0 & 0 & 0 & 0 & 0& 0 & 0 & 0\\
%			0 & &&& 0 & 0 & 0 & 0 & 0 & 0 & 0 & 0 & 0  \\
%			0 & 0 & 0 & 0 & &&& 0 & 0 & 0 & 0 & 0 & 0 \\
%			0 & 0 & 0 & 0 & & 2I_3& & 0 & 0 & 0 & 0 & 0 & 0\\
%			0 & 0 & 0 & 0 & &&& 0 & 0 & 0 & 0 & 0& 0 \\
%			0 & 0 & 0 & 0 & 0 & 0 & 0 &&& & 0 & 0 & 0 \\
%			0 & 0 & 0 & 0 & 0 & 0 & 0 & & I_3& & 0 & 0 & 0\\
%			0 & 0 & 0 & 0 & 0 & 0 & 0 & && & 0 & 0& 0 \\
%			0 &  &  &  &  &  &  & 0 & 0 & 0 & && \\
%			0 &  & 0 &  &  & 0 &  & 0 & 0 & 0 & &\sqrt{10}I_3&\\
%			0 &  &  &  &  &  &  & 0 & 0 & 0& &&
%	\end{array} } \right] \cr
= diag\left(1,1,1,1,2,2,2,1,1,1,\sqrt{10},\sqrt{10},\sqrt{10}\right).
\end{split}
\end{align}
Finally, Lemma \ref{S diff} gives
\begin{align}\label{nablaa}
\nabla_{(\rho_{\theta},U_{\theta},\Theta_{\theta},q_{\theta})}\mathcal{S}_{Pr}(\theta)\,\Big|_{\theta=0}=	\bigg( 1,v,  \bigg(\frac{|v|^2-3}{6}\bigg)\,\mathbb{I}_3, 0^3, 
	\frac{1-Pr}{5}v\bigg(\frac{|v|^2}{2}-\frac{5}{2}\bigg)\bigg)m,
\end{align}
where $\mathbb{I}_3$ denotes $(1,1,1)$.\newline

Now, we substitute \eqref{rugh}, \eqref{diag0} and \eqref{nablaa} into \eqref{S'} to derive the following expression for 
$\mathcal{S}_{Pr}'(0)$.
%\begin{align*}
%\left(\nabla_{(\rho_{\theta},U_{\theta},\Theta_{\theta},q_{\theta})}\mathcal{S}_{Pr}(\theta)\right)^{T}\bigg|_{\theta=0} =\left[
%1, v, \frac{|v|^2-3}{6}, \frac{|v|^2-3}{6}, \frac{|v|^2-3}{6}, 0, 0 ,0 , \frac{1-Pr}{5}\frac{|v|^2-5}{2}v \right]^Tm
%\end{align*}
\begin{align*}
\mathcal{S}_{Pr}'(0)&= \left(\int_{\mathbb{R}^3}f\sqrt{m} dv\right) m + \left(\int_{\mathbb{R}^3}fv\sqrt{m} dv\right) vm  +\sum_{1\leq i \leq 3}\left(\int_{\mathbb{R}^3}f\frac{v_i^2-1}{2}\sqrt{m} dv\right) \frac{|v|^2-3}{3}m \cr
&+\sum_{1\leq i \leq 3}\left(\int_{\mathbb{R}^3}fv_i\frac{|v|^2-5}{\sqrt{10}}\sqrt{m} dv \right)  (1-Pr)v_i\frac{|v|^2-5}{\sqrt{10}}m  \cr
&=P_cf\sqrt{m}+(1-Pr)P_{nc}f\sqrt{m}.\cr
\end{align*}

\noindent $\bullet$ Expression of the integral term: An explicit computation gives
\begin{align*}
\mathcal{S}_{Pr}''(\theta)&= \frac{d^2\mathcal{S}_{Pr}}{d\theta^2}(\rho_{\theta},\rho_{\theta} U_{\theta},G_{\theta},H_{\theta})	\cr
&=(\rho-1, \rho U, G, H)^T\left\{D^2_{(\rho_{\theta},\rho_{\theta} U_{\theta},G_{\theta},H_{\theta})}\mathcal{S}_{Pr}(\theta)\right\}(\rho-1, \rho U, G, H).
\end{align*}
Then, \eqref{rugh} yields 
\begin{align*}
\mathcal{S}_{Pr}''(\theta)&= \sum_{1 \leq i,j \leq 13} \left\{D^2_{(\rho_{\theta},\rho_{\theta} U_{\theta},G_{\theta},H_{\theta})}\mathcal{S}_{Pr}(\theta)\right\}_{ij}\langle f, e_i \rangle_{L^2_v}\langle f, e_j \rangle_{L^2_v}.
\end{align*}
This completes the proof.
\end{proof}
In the following lemma, we rewrite the second-order term $\mathcal{S}_{Pr}''(\theta)$ in a more tractable manner.
\begin{lemma}\label{non-linear poly} Each element of $D^2_{(\rho_{\theta},\rho_{\theta} U_{\theta},G_{\theta},H_{\theta})}\mathcal{S}_{Pr}(\theta)$ can be expressed in the following form:
\begin{align*}
\{D^2_{(\rho_{\theta},\rho_{\theta} U_{\theta},G_{\theta},H_{\theta})}\mathcal{S}_{Pr}(\theta)\}_{ij} =\frac{\mathcal{P}_{ij}(\rho_{\theta},U_{\theta},\Theta_{\theta},q_{\theta},(v_i-U_{\theta i}),1-Pr)}{\rho_{\theta}^{m}T_{\theta}^{n}}\mathcal{M}(\theta),
\end{align*}
where 
\begin{itemize}
\item $\mathcal{M}(\theta)$ is defined as 
\begin{align*}
\mathcal{M}(\theta)=\frac{\rho_{\theta}}{\sqrt{2\pi T_{\theta}}^3}\exp\left(-\frac{|v-U_{\theta}|^2}{2T_{\theta}}\right),
\end{align*}
with
\begin{align*}
3T_{\theta}=\sum_{1\leq i \leq 3 }\Theta_{\theta ii}.
\end{align*}
\item $\mathcal{P}$ is a generically defined polynomial of the following form : 
\begin{align*}
\mathcal{P}_{ij}(x_1,\cdots,x_n) = \sum_{k} a_kx_1^{k_1}\cdots x_n^{k_n},
\end{align*}
for a multi-index  $k=(k_1,\cdots k_n)$ where $k_i$ $(i=1,\cdots,n)$ is non-negative integers.
\item  $m$ and $n$ are non-negative integers that are not simultaneously zero at the same time.
\end{itemize}
\end{lemma}
\begin{proof}
Applying \eqref{changeofvar} twice,
\begin{multline*}
D^2_{(\rho_{\theta},\rho_{\theta} U_{\theta},G_{\theta},H_{\theta})}\mathcal{S}_{Pr}(\theta)=  \left( \frac{\partial(\rho_{\theta},\rho_{\theta} U_{\theta},G_{\theta},H_{\theta})} {\partial(\rho_{\theta},U_{\theta},\Theta_{\theta},q_{\theta})} \right)^{-1} \cr
\times \nabla_{(\rho_{\theta},U_{\theta},\Theta_{\theta},q_{\theta})}\left[\left( \frac{\partial(\rho_{\theta},\rho_{\theta} U_{\theta},G_{\theta},H_{\theta})} {\partial(\rho_{\theta},U_{\theta},\Theta_{\theta},q_{\theta})} \right)^{-1} \nabla_{(\rho_{\theta},U_{\theta},\Theta_{\theta},q_{\theta})} \mathcal{S}_{Pr}(\theta)\right].
\end{multline*}
For simplicity, we only consider the $(1,1)$ and $(1,2)$ components of $D^2_{(\rho_{\theta},\rho_{\theta} U_{\theta},G_{\theta},H_{\theta})}\mathcal{S}_{Pr}(\theta)$.
%\begin{align*}
%\left\{D^2_{(\rho_{\theta},\rho_{\theta} U_{\theta},G_{\theta},H_{\theta})}\mathcal{S}_{Pr}(\theta)\right\}_{11}&=\frac{\partial}{\partial \rho_{\theta}}\left(\frac{1}{\rho_{\theta}}\mathcal{M}(\theta)\right) =0
%\end{align*}
%\begin{align*}
%\left\{D^2_{(\rho_{\theta},\rho_{\theta} U_{\theta},G_{\theta},H_{\theta})}\mathcal{S}_{Pr}(\theta)\right\}_{12}&=\sum_{i=1}^{13}\left(\left( \frac{\partial(\rho_{\theta},\rho_{\theta} U_{\theta},G_{\theta},H_{\theta})} {\partial(\rho_{\theta},U_{\theta},\Theta_{\theta},q_{\theta})} \right)^{-1}\right)_{1i} \cr
%&\times \left(\nabla_{(\rho_{\theta},U_{\theta},\Theta_{\theta},q_{\theta})}\left[\left( \frac{\partial(\rho_{\theta},\rho_{\theta} U_{\theta},G_{\theta},H_{\theta})} {\partial(\rho_{\theta},U_{\theta},\Theta_{\theta},q_{\theta})} \right)^{-1} \nabla_{(\rho_{\theta},U_{\theta},\Theta_{\theta},q_{\theta})} \mathcal{S}_{Pr}(\theta)\right]\right)_{i2}.
%\end{align*}
Let us define the quantity in the second line as $B$:
\begin{align*}
B=\nabla_{(\rho_{\theta},U_{\theta},\Theta_{\theta},q_{\theta})}\left[\left( \frac{\partial(\rho_{\theta},\rho_{\theta} U_{\theta},G_{\theta},H_{\theta})} {\partial(\rho_{\theta},U_{\theta},\Theta_{\theta},q_{\theta})} \right)^{-1} \nabla_{(\rho_{\theta},U_{\theta},\Theta_{\theta},q_{\theta})} \mathcal{S}_{Pr}(\theta)\right].
\end{align*}
The $(1,1)$ component of $D^2_{(\rho_{\theta},\rho_{\theta} U_{\theta},G_{\theta},H_{\theta})}\mathcal{S}_{Pr}(\theta)$ is determined by inner product of first row of $J^{-1}$ and first column of $B$. As we can see in Lemma \ref{Jaco} (2), the components of the first row of $J^{-1}$ are all zeros except the first component. Thus we only need to compute the $(1,1)$ component of $B$:
\begin{align*}
(B)_{11} &= \frac{\partial}{\partial \rho_{\theta}}\left[\left( \frac{\partial(\rho_{\theta},\rho_{\theta} U_{\theta},G_{\theta},H_{\theta})} {\partial(\rho_{\theta},U_{\theta},\Theta_{\theta},q_{\theta})} \right)^{-1} \nabla_{(\rho_{\theta},U_{\theta},\Theta_{\theta},q_{\theta})} \mathcal{S}_{Pr}(\theta)\right]_1 =\frac{\partial}{\partial \rho_{\theta}}\left(\frac{\partial \mathcal{S}_{Pr}(\theta) }{\partial \rho_{\theta}} \right).
\end{align*}
Applying the computation in Lemma \ref{S diff} (1) gives 
\begin{align*}
(B)_{11} &= \frac{\partial}{\partial \rho_{\theta}}\left(\frac{1}{\rho_{\theta}}\mathcal{M}(\theta) \right)= 0.
\end{align*}
Thus we have
\begin{align*}
\left\{D^2_{(\rho_{\theta},\rho_{\theta} U_{\theta},G_{\theta},H_{\theta})}\mathcal{S}_{Pr}(\theta)\right\}_{11}&=0.
\end{align*}
Similarly $(1,2)$ component of $D^2_{(\rho_{\theta},\rho_{\theta} U_{\theta},G_{\theta},H_{\theta})}\mathcal{S}_{Pr}(\theta)$ is determined by the inner product of the first row of $J^{-1}$ and second column of $B$. Since the components of first row of $J^{-1}$ are all zeros except the first component, we only need to compute $(1,2)$ component of $B$,
\begin{align*}
(B)_{12} &= \frac{\partial}{\partial \rho_{\theta}}\left[\left( \frac{\partial(\rho_{\theta},\rho_{\theta} U_{\theta},G_{\theta},H_{\theta})} {\partial(\rho_{\theta},U_{\theta},\Theta_{\theta},q_{\theta})} \right)^{-1} \nabla_{(\rho_{\theta},U_{\theta},\Theta_{\theta},q_{\theta})} \mathcal{S}_{Pr}(\theta)\right]_2.
\end{align*}
The second component in the square brackets is inner product of second row of $J^{-1}$ and $\nabla_{(\rho_{\theta},U_{\theta},\Theta_{\theta},q_{\theta})} \mathcal{S}_{Pr}(\theta)$. Thus we have 
\begin{align*}
(B)_{12}&=\frac{\partial}{\partial \rho_{\theta}}\left[-\frac{U_{\theta 1}}{\rho_{\theta}}\frac{\partial \mathcal{S}_{Pr}(\theta)}{\partial \rho_{\theta}}+\frac{1}{\rho_{\theta}}\frac{\partial \mathcal{S}_{Pr}(\theta)}{\partial U_{\theta 1}} \right].
\end{align*}
Combining the above computations, we obtain
\begin{align*}
\left\{D^2_{(\rho_{\theta},\rho_{\theta} U_{\theta},G_{\theta},H_{\theta})}\mathcal{S}_{Pr}(\theta)\right\}_{12}
&= \frac{\partial}{\partial \rho_{\theta}}\left(-\frac{U_{\theta 1}}{\rho_{\theta}}\frac{\partial \mathcal{S}_{Pr}(\theta)}{\partial \rho_{\theta}}+\frac{1}{\rho_{\theta}}\frac{\partial \mathcal{S}_{Pr}(\theta)}{\partial U_{\theta 1}} \right).
\end{align*}
Applying Lemma \ref{S diff}, it is equal to 
\begin{align*}
\left\{D^2_{(\rho_{\theta},\rho_{\theta} U_{\theta},G_{\theta},H_{\theta})}\mathcal{S}_{Pr}(\theta)\right\}_{12}&=\frac{\partial}{\partial \rho_{\theta}}\left(-\frac{U_{\theta 1}}{\rho_{\theta}^2}\mathcal{M}(\theta)+\frac{1}{\rho_{\theta}}\frac{\partial \mathcal{S}_{Pr}(\theta)}{\partial U_{\theta 1}} \right) \cr
&=\left(\frac{U_{\theta 1}}{\rho_{\theta}^3}\mathcal{M}(\theta)-\frac{1}{\rho_{\theta}
^2}\frac{\partial \mathcal{S}_{Pr}(\theta)}{\partial U_{\theta 1}}+\frac{1}{\rho_{\theta}}\frac{\partial^2 \mathcal{S}_{Pr}(\theta)}{\partial \rho_{\theta}\partial U_{\theta 1}} \right),
\end{align*}
where
\begin{align*}
\frac{\partial \mathcal{S}_{Pr}(\theta)}{\partial U_{\theta 1}}&=\frac{v_i-U_{\theta i}}{T_{\theta}}\left[1+\frac{1-Pr}{5}\frac{q_{\theta}\cdot (v-U_{\theta})}{\rho_{\theta} T_{\theta}^2}\left(\frac{|v-U_{\theta}|^2}{2T_{\theta}}-\frac{5}{2}\right)\right]\mathcal{M}(\theta) \cr
&-\left[\frac{1-Pr}{5}\frac{q_{\theta i}}{\rho_{\theta} T_{\theta}^2}\left(\frac{|v-U_{\theta}|^2}{2T_{\theta}}-\frac{5}{2}\right)+\frac{1-Pr}{5}\frac{q_{\theta}\cdot (v-U_{\theta})}{\rho_{\theta} T_{\theta}^2}\left(\frac{v_i-U_{\theta i}}{T_{\theta}}\right)\right]\mathcal{M}(\theta),
\end{align*}
and 
\begin{align*}
\frac{\partial^2 \mathcal{S}_{Pr}(\theta)}{\partial \rho_{\theta}\partial U_{\theta 1}}&=\frac{v_i-U_{\theta i}}{\rho_\theta T_\theta}\mathcal{M}(\theta).
\end{align*}
This shows that $(1,2)$ component of $D^2_{(\rho_{\theta},\rho_{\theta} U_{\theta},G_{\theta},H_{\theta})}\mathcal{S}_{Pr}(\theta)$ follows the proposed form of this Lemma. Other terms are similar. We omit it.
\end{proof}

It remains to linearize the collision frequency.
\begin{lemma} The general collision frequency \eqref{tau} is linearized as follows: 
\begin{align*}
\frac{1}{\tau} =\frac{1}{\tau_0}\left( 1  + \int_0^1 A_1(\theta) d\theta \left(\int_{\mathbb{R}^3} f \sqrt{m} dv\right)+  \int_0^1 A_2(\theta) d\theta \left(\int_{\mathbb{R}^3} f \frac{|v|^2-3}{\sqrt{6}}\sqrt{m} dv\right)\right),
\end{align*}
where
\begin{align*}
A_1(\theta)&=\left(1+\frac{|U_{\theta}|^2-3T_{\theta}+3}{3\rho_{\theta}}\right)\eta\rho^{\eta-1}_{\theta} T^w_{\theta}, \quad 
A_2(\theta)= w\rho^{\eta}_{\theta} T^{w-1}_{\theta}.
\end{align*}
\end{lemma}
\begin{proof}
As in the proof of Proposition \ref{linear}, we first define the transition of the macroscopic fields:
\begin{align*}
\rho_{\theta} &= \theta \rho +(1-\theta), \quad
\rho_{\theta}U_{\theta} = \theta \rho U, \quad
K_{\theta}= \theta K,
\end{align*}
where 
\begin{align*}
K_{\theta} &= \frac{3\rho_{\theta} T_{\theta} + \rho_{\theta} |U_{\theta}|^2-3\rho_{\theta}}{\sqrt{6}},
\end{align*}
and the transitional collision frequency depending on $(\rho_{\theta},U_{\theta},T_{\theta})$:
\begin{align*}
A(\theta) = \frac{1}{\tau_0}\rho^{\eta}_{\theta} T^w_{\theta}.
\end{align*}
We then expand $A(\theta)$ as 
\begin{align*}
A(1) = A(0)+ \int_0^1 A'(\theta)d\theta.
\end{align*}
Then the chain rule gives 
\begin{align}\label{c0}
A'(\theta) = \frac{1}{\tau_0}\left(\frac{d \rho_{\theta}}{d\theta },\frac{d \rho_{\theta}U_{\theta}}{d\theta },\frac{d K_{\theta}}{d\theta }\right) \left[\frac{\partial(\rho_{\theta},\rho_{\theta}U_{\theta},K_{\theta})}{\partial(\rho_{\theta},U_{\theta},T_{\theta})}\right]^{-1}\nabla_{\left(\rho_{\theta},U_{\theta},T_{\theta}\right)}\rho^{\eta}_{\theta} T^w_{\theta},
\end{align}
where each component of \eqref{c0} can be computed by the following three equality:
\begin{align}\label{c1}
\begin{split}
\left(\frac{d \rho_{\theta}}{d\theta },\frac{d \rho_{\theta}U_{\theta}}{d\theta },\frac{d K_{\theta}}{d\theta }\right)% &= (\rho-1,\rho U,K) \cr
&= \left( \int_{\mathbb{R}^3} f \sqrt{m} dv,\int_{\mathbb{R}^3} f v\sqrt{m} dv, \int_{\mathbb{R}^3} f \frac{|v|^2-3}{\sqrt{6}}\sqrt{m} dv\right) ,
\end{split}
\end{align}
and
\begin{align}\label{c2}
\begin{split}
\left[\frac{\partial(\rho_{\theta},\rho_{\theta}U_{\theta},K_{\theta})}{\partial(\rho_{\theta},U_{\theta},T_{\theta})}\right]^{-1} = \left[ \begin{array}{ccc} 1 & 0 & 0 \\
-\frac{U_{\theta}}{\rho_{\theta}} & \frac{1}{\rho_{\theta}}I_3 & 0  \\
\frac{|U_{\theta}|^2-3T_{\theta}+3}{3\rho_{\theta}} & -\frac{2}{3}\frac{U_{\theta}}{\rho_{\theta}}& \sqrt{\frac{2}{3}}\frac{1}{\rho_{\theta}} \end{array} \right],
\end{split}
\end{align}
and
\iffalse
\begin{align}
J^{-1} = \left[ {\begin{array}{cccccc}
1 & 0 & 0 & 0 & 0  \\ 
-\frac{U_1}{\rho} & \frac{1}{\rho} & 0 & 0 & 0 \\
-\frac{U_2}{\rho} & 0 & \frac{1}{\rho} & 0 & 0 \\
-\frac{U_3}{\rho} & 0 & 0 & \frac{1}{\rho} & 0 \\
\frac{|U|^2-3T+3}{3\rho} & -\frac{2}{3}\frac{U_1}{\rho} & -\frac{2}{3}\frac{U_2}{\rho} & -\frac{2}{3}\frac{U_3}{\rho} & \sqrt{\frac{2}{3}}\frac{1}{\rho}
\end{array} } \right].
\end{align}
\fi
\begin{align}\label{c3}
\nabla_{\left(\rho_{\theta},U_{\theta},T_{\theta}\right)}\rho^{\eta}_{\theta} T^w_{\theta} = \left(\eta\rho^{\eta-1}_{\theta} T^w_{\theta}, 0^3, w\rho^{\eta}_{\theta} T^{w-1}_{\theta} \right).
\end{align}
Substituting \eqref{c1}-\eqref{c3} into \eqref{c0}, we get the desired result.
%\begin{align*}
%\frac{1}{\tau} =\frac{1}{\tau_0}+ \int_0^1 A_1(\theta) d\theta \int_{\mathbb{R}^3} f \sqrt{m} dv+  \int_0^1 A_2(\theta) d\theta \int_{\mathbb{R}^3} f \frac{|v|^2-3}{\sqrt{6}}\sqrt{m} dv
%\end{align*}
\end{proof}

\subsection{Linearized Shakhov model} We are ready to derive the linearized Shakhov model. We insert $F=m+\sqrt{m}f$ in \eqref{Shakhov} and apply Proposition \ref{linear} and Lemma \ref{non-linear poly}, to get
\begin{align}\label{perturb}
\begin{split}
\partial_tf+v\cdot \nabla_xf &= \frac{1}{\tau_0}L_{Pr}f+\Gamma(f) , \cr
f(x,v,0)&=f_0(x,v),
\end{split}
\end{align}
where $f_0(x,v)= (F_0(x,v)-m)/\sqrt{m}$. The linear operator $L_{Pr}$ is 
\begin{align*}
L_{Pr}f = P_{Pr}f-f,
\end{align*}
where the projection operator $P_{Pr}$ is defined in Definition \ref{P_{Pr}f}, and the non-linear term is decomposed as
\begin{align}\label{non linear}
\Gamma(f) = \sum_{i=1}^3 \Gamma_i(f),
\end{align}
where
\begin{align*}
\Gamma_1(f)&= \left(\frac{1}{\tau}-\frac{1}{\tau_0} \right)L_{Pr}f, \cr
\Gamma_2(f)&= \frac{1}{\tau_0}\frac{1}{\sqrt{m}}\sum_{1 \leq i,j \leq 13} \int_0^1 \left\{D^2_{(\rho_{\theta},\rho_{\theta} U_{\theta},G_{\theta},H_{\theta})}\mathcal{S}_{Pr}(\theta)\right\}_{ij}(1-\theta)d\theta \langle f, e_i \rangle_{L^2_v}\langle f, e_j \rangle_{L^2_v},\cr
\Gamma_3(f)&= \left(\frac{1}{\tau}-\frac{1}{\tau_0} \right)\frac{1}{\sqrt{m}}\sum_{1 \leq i,j \leq 13} \int_0^1 \left\{D^2_{(\rho_{\theta},\rho_{\theta} U_{\theta},G_{\theta},H_{\theta})}\mathcal{S}_{Pr}(\theta)\right\}_{ij}(1-\theta)d\theta \langle f, e_i \rangle_{L^2_v}\langle f, e_j \rangle_{L^2_v}.
\end{align*}

The conservation laws \eqref{conserv} are rewritten as follows:
\begin{align}\label{conservf}
\begin{split}
\int_{\mathbb{T}^3\times\mathbb{R}^3} f(x,v,t)\sqrt{m} dvdx &= \int_{\mathbb{T}^3\times\mathbb{R}^3} f_0(x,v)\sqrt{m} dvdx, \cr
\int_{\mathbb{T}^3\times\mathbb{R}^3} f(x,v,t)v\sqrt{m} dvdx &= \int_{\mathbb{T}^3\times\mathbb{R}^3} f_0(x,v)v\sqrt{m} dvdx, \cr
\int_{\mathbb{T}^3\times\mathbb{R}^3} f(x,v,t)|v|^2\sqrt{m} dvdx &= \int_{\mathbb{T}^3\times\mathbb{R}^3} f_0(x,v)|v|^2\sqrt{m} dvdx.
\end{split}
\end{align}

\subsection{Properties of the linear term}
In this section, we establish the coercivity of $L_{Pr}$. We observe the dichotomy in the dissipative nature of $L_{Pr}$ 
between the case $Pr>0$ and $Pr=0$. We first define the following basis: 
\begin{align}\label{bare}
\bar{e}_1&=\sqrt{m}, \quad \bar{e}_{i+1}=v_i\sqrt{m}, \quad \bar{e}_{5}=\frac{|v|^2-3}{\sqrt{6}}\sqrt{m}, \quad 
\bar{e}_{i+5} = \frac{v_i|v|^2-5v_i}{\sqrt{10}}\sqrt{m},
\end{align}
for $i=1,2,3$, and write $P_{Pr}f$ (See Definition \ref{P_{Pr}f}) as
\begin{align*}
P_{Pr}f&=P_cf+(1-Pr)P_{nc}f\cr
&=\sum_{1 \leq i \leq 5} \langle f, \bar{e}_i \rangle_{L^2_v} \bar{e}_i+ (1-Pr) \sum_{6 \leq i \leq 8} \langle f, \bar{e}_i \rangle_{L^2_v} \bar{e}_i.
\end{align*}
\begin{lemma}\label{P0Ps} The projection operator $P_c$ and $P_{nc}$ satisfy the following properties:  
\begin{enumerate}
\item $P_c$ and $P_{nc}$ are orthonormal projection.
\begin{align*}
P_c^2=P_c,\quad  P_{nc}^2=P_{nc}.
\end{align*}
\item $P_c$ and $P_{nc}$ are orthogonal.
\begin{align*}
P_c \perp P_{nc}.
\end{align*}
\end{enumerate}
\end{lemma}
\begin{proof}
The first statement follows from that each of the following:
\[
\{\bar{e}_1,\,\bar{e}_2,\,\bar{e}_3,\,\bar{e}_4,\,\bar{e}_5\},
\]
and
\[
\{\bar{e}_6,\,\bar{e}_7,\,\bar{e}_8\},
\]
forms an orthonormal basis, while the second statement is derived from  
\begin{align*}
\langle \bar{e}_i,\bar{e}_{j} \rangle_{L^2_v} = 0\quad (1\leq i\leq 5,~~6\leq j\leq 8),
\end{align*}
which can be checked through a direct computation.
\end{proof}
In the following proposition, we prove the main result of this section, namely the dichotomy between $Pr>0$ and $Pr=0$ in the dissipative property of the linearized Shakhov operator. We note that the degeneracy of the estimate is stronger in the case of $Pr=0$.
\begin{proposition}\label{dichotomy} \noindent$(1)$ In the case $Pr>0$,  $L_{Pr}f$ satisfies 
\begin{align*}
\langle L_{Pr}f,f \rangle_{L^2_v} \leq -\min\{Pr,1\} \|(I-P_c)f \|_{L^2_v}^2.
\end{align*}
\noindent$(2)$ When $Pr= 0$, $L_{Pr}$ satisfies 
	\begin{align*}
		\langle L_{Pr}f,f \rangle_{L^2_{x,v}} = -\|(I-P_{Pr})f\|_{L^2_{x,v}}^2= -\|(I-P_c-P_{nc})f\|_{L^2_{x,v}}^2.
	\end{align*}
\end{proposition}

\begin{proof}
\noindent(1) By an explicit computation, we have 
\begin{align*}
\langle L_{Pr}f,f \rangle_{L^2_v} &= \langle P_cf-f+(1-Pr)P_{nc}f,f \rangle_{L^2_v} \cr
&= \langle P_cf-f,f \rangle_{L^2_v} + (1-Pr) \langle P_{nc}f,f \rangle_{L^2_v}.
\end{align*}
Since $\bar{e}_1,\cdots,\bar{e}_5$ constitute an orthonormal basis, we have
\begin{align*}
\langle P_cf-f,f \rangle_{L^2_v} &=- \langle (I-P_c)f,f \rangle_{L^2_v}\cr
&=- \langle (I-P_c)f,(I-P_c)f \rangle_{L^2_v}-\langle (I-P_c)f,P_cf \rangle_{L^2_v} \cr
&=-\|(I-P_c)f\|_{L^2_v}^2,
\end{align*}
which implies
\begin{align}\label{L_{Pr}f,f}
\begin{split}
\langle L_{Pr}f,f \rangle_{L^2_v} &=  -\|(I-P_c)f\|_{L^2_v}^2 + (1-Pr) \langle P_{nc}f,f \rangle_{L^2_v}.
\end{split}
\end{align}
Applying the property $P_{nc}\perp P_c$ in Lemma \ref{P0Ps} (2), we have 
\begin{align*}
\langle P_{nc}f,f \rangle_{L^2_v} &=\langle P_{nc}(P_cf+(I-P_c)f),P_cf+(I-P_c)f \rangle_{L^2_v} \cr
&= \langle P_{nc}(I-P_c)f,(I-P_c)f \rangle_{L^2_v} \cr
&= \| P_{nc}(I-P_c)f\|_{L^2_v}^2,
\end{align*}
which gives 
\begin{align}\label{Psf,f}
0 \leq \langle P_{nc}f,f \rangle_{L^2_v} \leq \| (I-P_c)f\|_{L^2_v}^2.
\end{align}
When $0<Pr\leq 1$, substituting \eqref{Psf,f} in \eqref{L_{Pr}f,f} yields 
\begin{align*}
	\langle L_{Pr}f,f \rangle_{L^2_v} \leq -Pr \| (I-P_c)f\|_{L^2_v}^2.
\end{align*}
In the case $1<Pr$, since $\langle P_{nc}f,f \rangle_{L^2_v}$ is non-negative, we can ignore the second term in the R.H.S of \eqref{L_{Pr}f,f}, to obtain
\begin{align*}
	\langle L_{Pr}f,f \rangle_{L^2_v} \leq - \| (I-P_c)f\|_{L^2_v}^2.
\end{align*}
We combine the above two inequalities to get the desired result.\newline

\noindent(2) Since $P_{Pr}f=P_cf+P_{nc}f$ is a projection operator onto the space spanned by the $8$-dimensional orthonormal basis $\{\bar{e}_i\}_{1\leq i \leq 8 }$, we have
\begin{align*}
	\langle L_{Pr}f,f \rangle_{L^2_{x,v}} &=- \langle (I-P_{Pr})f,f \rangle_{L^2_{x,v}}\cr
	&=- \langle (I-P_{Pr})f,(I-P_{Pr})f \rangle_{L^2_{x,v}}-\langle (I-P_{Pr})f,P_{Pr}f \rangle_{L^2_{x,v}} \cr
	&=-\|(I-P_{Pr})f\|_{L^2_{x,v}}^2.
\end{align*}
\end{proof}

\begin{lemma}\label{ker} When $Pr>0$, the kernel of the linear operator $L_{Pr}$ is given by the following 5-dimensional space:
\begin{align*}
Ker{L} = span\{\sqrt{m},v\sqrt{m},|v|^2\sqrt{m}\},
\end{align*}
while in the case of $Pr=0$,  $L_{Pr}$ has a larger kernel spanned by the following 8 functions: 
\begin{align*}
	Ker{L}= span\{\sqrt{m},v\sqrt{m},|v|^2\sqrt{m},v|v|^2\sqrt{m} \}.
\end{align*}
\end{lemma}
\begin{proof}
This follows directly from Proposition \ref{dichotomy}.
\end{proof}

\section{Local solution}
In this section, we construct the local-in-time classical solution. We first estimate macroscopic fields $\rho$, $U$, $\Theta$ and $q$.
\subsection{Estimates for the macroscopic fields}
%In this section, our goal is to construct the local-in-time classical solution. For that, we first estimate macroscopic fields $\rho$, $U$, $\Theta$ and $q$.
%including the stress tensor $\Theta$ and the heat flux $q$. %After that we estimates the non-linear term.%The second is estimates of non-linear term associated with the macroscopic fields. At last, we construct the local-in-time classical solution for perturbation $f$. %The non-linear term $\Gamma(f)$ consists of several macroscopic quantities.
\begin{lemma}\label{macro} Let $N\geq3$. For  sufficiently small $\mathcal{E}(t)$, there exist  positive constants $C$ such that
\begin{align*}
&(1) \ |\rho(x,t)-1|\leq C\sqrt{\mathcal{E}(t)},	\cr
&(2) \ |U(x,t)|\leq C\sqrt{\mathcal{E}(t)},	\cr
&(3) \ |\Theta_{ij}(x,t)-\delta_{ij}| \leq C\sqrt{\mathcal{E}(t)}, \cr
&(4) \ |q_i(x,t)|\leq C\sqrt{\mathcal{E}(t)} ,
\end{align*}
for $1\leq i,j \leq 3$.
\end{lemma}
\begin{proof} (1) Since
\begin{align}\label{rho}
\rho=\int_{\mathbb{R}^3}m+\sqrt{m}f dv = 1+\int_{\mathbb{R}^3}\sqrt{m}f dv .
\end{align}
The H\"{o}lder inequality and the Sobolev embedding $H^2 \subset\subset L^{\infty}$ give
\begin{align*}
|\rho-1| \leq C\sup_{x\in\mathbb{T}^3}\|f\|_{L^2_v}\leq \sum_{|\alpha|\leq 2}\| \partial^{\alpha} f \|_{L^2_{x,v}} \leq C\sqrt{\mathcal{E}(t)}.
\end{align*}
(2) We write the bulk velocity $U$ as
\begin{align*}
	\rho U=\int_{\mathbb{R}^3}(m+\sqrt{m}f)v dv = \int_{\mathbb{R}^3}\sqrt{m}fv dv.
\end{align*}
Then, the H\"{o}lder inequality and the Sobolev embedding, together with the lower bound of $\rho$ in (1) yields
\begin{align*}
|U| \leq \frac{\left(\int_{\mathbb{R}^3}|f|^2dv\right)^\frac{1}{2}\left(\int_{\mathbb{R}^3}m|v|^2dv\right)^\frac{1}{2}}{1-C\sqrt{\mathcal{E}(t)}}\leq \frac{C\sqrt{\mathcal{E}(t)}}{1-C\sqrt{\mathcal{E}(t)}} \leq C\sqrt{\mathcal{E}(t)}.
\end{align*}
(3) For the estimate of $\Theta$, we recall the computation in $(\ref{viv2})_1$: 
\begin{align*}
\rho\Theta_{ij}&=\int_{\mathbb{R}^3}F(v_i-U_i)(v_j-U_j) dv  = \int_{\mathbb{R}^3}(m+\sqrt{m}f)v_iv_j dv -\rho U_iU_j.
\end{align*}
When $i=j$, we apply the H\"{o}lder inequality and the estimates in (1) and (2) to get
\begin{align*}
|\rho \Theta_{ii}-1|&\leq \left(\int_{\mathbb{R}^3}|f|^2dv\right)^\frac{1}{2}\left(\int_{\mathbb{R}^3}mv_i^4dv\right)^\frac{1}{2}+C(1+C\sqrt{\mathcal{E}(t)})\mathcal{E}(t) \leq C\sqrt{\mathcal{E}(t)},
\end{align*}
%Using the boundedness of $\rho$ in (1), we have 
%\begin{align*}
%\frac{1-C\sqrt{\mathcal{E}(t)}}{1+C\sqrt{\mathcal{E}(t)}}	\leq \Theta_{ii} %&\leq\frac{1+C\sqrt{\mathcal{E}(t)}}{1-C\sqrt{\mathcal{E}(t)}},
%\end{align*}
which gives
\begin{align*}
| \Theta_{ii}-1|&\leq C\sqrt{\mathcal{E}(t)}.
\end{align*}
When $i \neq j $, since $\int_{\mathbb{R}^3}mv_iv_j dv=0$, we get
\begin{align*}
|\Theta_{ij}|&\leq \frac{\left(\int_{\mathbb{R}^3}|f|^2dv\right)^\frac{1}{2}\left(\int_{\mathbb{R}^3}mv_i^2v_j^2dv\right)^\frac{1}{2}+C(1+C\sqrt{\mathcal{E}(t)})\mathcal{E}(t)}{1-C\sqrt{\mathcal{E}(t)}} \leq C\sqrt{\mathcal{E}(t)}.
\end{align*}
(4) Recall the computation in $(\ref{viv2})_2$ that 
\begin{align*}
q_i	 &= \int_{\mathbb{R}^3}Fv_i|v|^2dv -\sum_{1\leq j \leq 3}2U_j\rho\Theta_{ij}- \rho U_i|U|^2 - U_i \rho  \sum_{1\leq i \leq 3 }\Theta_{ii}.
\end{align*}
Combining this with the above estimates (1)-(3), we have
\begin{align*}
|q_i|	 &\leq C\|f\|_{L^2_v}+C\sqrt{\mathcal{E}(t)} \leq C\sqrt{\mathcal{E}(t)}.
\end{align*}
\end{proof}
\begin{lemma}\label{macro diff} Let $N=|\alpha| \geq 1$. For sufficiently small $\mathcal{E}(t)$, there exist positive constants $C$ and $C_{\alpha}$ such that
\begin{align*}
&(1) \ |\partial^{\alpha}\rho(x,t)| \leq C\|\partial^{\alpha}f\|_{L^2_v},	\cr
&(2) \ |\partial^{\alpha}U(x,t)| \leq C_{\alpha}\sum_{|\alpha_1|\leq|\alpha|}\|\partial^{\alpha_1}f\|_{L^2_v},	\cr
&(3) \ |\partial^{\alpha}\Theta_{ij}(x,t)| \leq C_{\alpha}\sum_{|\alpha_1|\leq|\alpha|}\|\partial^{\alpha_1}f\|_{L^2_v} ,	\cr
&(4) \ |\partial^{\alpha}q_i(x,t)|  \leq C_{\alpha}\sum_{|\alpha_1|\leq|\alpha|}\|\partial^{\alpha_1}f\|_{L^2_v},
\end{align*}
for $1\leq i,j \leq 3$.
\end{lemma}
\begin{proof} (1) Taking $\partial^{\alpha}$ on \eqref{rho} and applying the H\"{o}lder inequality give
\begin{align*}
|\partial^{\alpha}\rho| = \bigg|\int_{\mathbb{R}^3}\sqrt{m}\partial^{\alpha}f dv\bigg| \leq  C\|\partial^{\alpha}f\|_{L^2_v} .
\end{align*}
(2) Similarly, we have
\begin{align*}
|\partial^{\alpha}U|&= \bigg| \partial^{\alpha}\frac{\int_{\mathbb{R}^3}\sqrt{m}fv dv}{\rho }\bigg|\\
&\leq C_{\alpha} \sum_{\alpha_1+\alpha_2=\alpha}\bigg|\int_{\mathbb{R}^3}\sqrt{m}\partial^{\alpha_1}fv dv\bigg|\bigg|\partial^{\alpha_2}\frac{1}{\rho}\bigg|\\
&\leq C_{\alpha}\sum_{\alpha_1+\alpha_2=\alpha}\|\partial^{\alpha_1}f\|_{L^2_v}\bigg|\partial^{\alpha_2}\frac{1}{\rho}\bigg|.
\end{align*}
We then use the boundedness of $\rho$ and $\partial^{\alpha}\rho$ in Lemma \ref{macro}, and the estimate (1) of this lemma, and apply the Sobolev embedding $H^2 \subset\subset L^{\infty}$ to obtain 
\begin{align*}
\bigg|\partial^{\alpha_2}\frac{1}{\rho}\bigg| &\leq \left( \displaystyle{\prod_{\sum|\alpha_{2i}|\leq|\alpha_2|}|\partial^{\alpha_{2i}}\rho|} \right) \left(\sum_{0\leq n \leq  |\alpha|}\bigg|\frac{1}{\rho}\bigg|^{n+1}\right) \cr 
&\leq \left(C\sqrt{\mathcal{E}(t)}\sum_{|\alpha_2|-1 \leq |\alpha_{2i}|\leq |\alpha_2|}\|\partial^{\alpha_{2i}}f\|_{L^2_v}\right)\left(\sum_{0\leq n \leq  |\alpha|}\bigg|\frac{1}{1-C\sqrt{\mathcal{E}(t)}}\bigg|^{n+1}\right)\cr
&\leq C_{\alpha}\sum_{|\alpha_1|\leq|\alpha|}\|\partial^{\alpha_1}f\|_{L^2_v},
\end{align*}
which gives the desired result. \newline
(3) Taking $\partial^{\alpha}$ on $\Theta_{ij}$ gives 
%\begin{align*}
%\Theta_{ij}&=\frac{\int_{\mathbb{R}^3}(m+\sqrt{m}f)v_iv_j dv}{\rho} - U_iU_j
%\end{align*}
\begin{align*}
|\partial^{\alpha}\Theta_{ij}|&\leq\bigg|\partial^{\alpha}\frac{\int_{\mathbb{R}^3}(m+\sqrt{m}f)v_iv_j dv}{\rho}\bigg| + \partial^{\alpha}(U_iU_j).
\end{align*}
Then by the Sobolev embedding $H^2 \subset\subset L^{\infty}$, 
\begin{align*}
	|\partial^{\alpha}(U_iU_j)|&\leq \sum_{|\alpha_1|+|\alpha_2|=|\alpha|}|\partial^{\alpha_1}U_i||\partial^{\alpha_1}U_j| \leq C_{\alpha}\sqrt{\mathcal{E}(t)} \| \partial^{\alpha}f\|_{L^2_v}.
\end{align*}
For sufficiently small $\mathcal{E}(t)$, the H\"{o}lder inequality gives
\begin{align*}
|\partial^{\alpha}\Theta_{ij}|&\leq C_{\alpha}\sum_{|\alpha_1|\leq|\alpha|}\|\partial^{\alpha_1}f\|_{L^2_v}.
\end{align*}
(4) Similarly, taking $\partial^{\alpha}$ on $q$ gives 
\begin{align*}
\partial^{\alpha}q_i	 &= \int_{\mathbb{R}^3}\partial^{\alpha}Fv_i|v|^2dv -2\sum_{1\leq j \leq 3}\partial^{\alpha}(\rho U_j\Theta_{ij})- \partial^{\alpha}( \rho U_i|U|^2) - \partial^{\alpha}\left(U_i \rho  \sum_{1\leq i \leq 3 }\Theta_{ii}\right).
\end{align*}
The previous results and the H\"{o}lder inequality yields 
\begin{align*}
|\partial^{\alpha}q_i|	 &\leq  \int_{\mathbb{R}^3}v_i|v|^2\sqrt{m}\partial^{\alpha}fdv+C_{\alpha}\mathcal{E}(t)\sum_{|\alpha_1|\leq|\alpha|}\|\partial^{\alpha_1}f\|_{L^2_v} \leq C_{\alpha}\sum_{|\alpha_1|\leq|\alpha|}\|\partial^{\alpha_1}f\|_{L^2_v}.
\end{align*}
\end{proof}
We also estimate the macroscopic fields depending on $\theta$. %Since the transition of macroscopic fields $\rho_{\theta}$, $U_{\theta}$, $\Theta_{\theta}$ and $q_{_{\theta}}$ are defined from the relation \eqref{transmacro}, we have to find explicit form of that. 
\begin{lemma}\label{macrotheta} For sufficiently small $\mathcal{E}(t)$, there exist positive constants $C$ such that
\begin{align*}
	&(1) \ |\rho_{\theta}(x,t)-1|\leq C\sqrt{\mathcal{E}(t)},	\cr
	&(2) \ |U_{\theta}(x,t)|\leq C\sqrt{\mathcal{E}(t)},	\cr
	&(3) \ |\Theta_{\theta ij}(x,t)-\delta_{ij}| \leq C\sqrt{\mathcal{E}(t)}, \cr
	&(4) \ |q_{_{\theta} i}(x,t)|\leq C\sqrt{\mathcal{E}(t)} ,
\end{align*}
for $1\leq i,j \leq 3$.
\end{lemma}
\begin{proof}
(1) Since $0\leq\theta \leq 1$, we have from Lemma \ref{macro} that 
\begin{align*}
|\rho_{\theta}-1|=\theta|\rho-1| \leq  C\sqrt{\mathcal{E}(t)}.
\end{align*}
(2) Applying the estimate (1) above and Lemma \ref{macro} gives
\begin{align*}
|U_{\theta}| = \bigg|\frac{\theta \rho U}{\rho_{\theta}} \bigg| \leq C\sqrt{\mathcal{E}(t)}.
\end{align*}
(3) From the definition of $G$ in \eqref{G}, we see that
\begin{align*}
\left(\rho_{\theta} \Theta_{\theta ii}+\rho_{\theta} U_{\theta i}^2-\rho_{\theta}\right) = \theta\left(\rho \Theta_{ii}+\rho U_i^2-\rho\right),
\end{align*}
for $i=j$ case. Applying Lemma \ref{macro} and estimates of (1) and (2) of this lemma, we have
\begin{align}\label{Thetaii}
|\Theta_{\theta ii}-1| = \bigg| \theta\frac{\rho \Theta_{ii}+\rho U_i^2-\rho}{\rho_{\theta}} -  U_{\theta i}^2 \bigg| \leq C\sqrt{\mathcal{E}(t)}.
\end{align}
When $i\neq j$, the definition $G_{\theta}=\theta G$ implies 
\begin{align*}
	\rho_{\theta}\Theta_{\theta ij}+\rho_{\theta} U_{\theta i}U_{\theta j}=\theta (\rho\Theta_{ij}+\rho U_iU_j).
\end{align*}
Therefore,
\begin{align}\label{Thetaij}
|\Theta_{\theta ij}|=\bigg|\theta \frac{\rho\Theta_{ij}+\rho U_iU_j}{\rho_{\theta}} -U_{\theta i}U_{\theta j}\bigg| \leq C\sqrt{\mathcal{E}(t)}.
\end{align}
(4) The definition of $H_{\theta}$ in \eqref{transmacro} with the definition of $H$ in \eqref{H} implies
\begin{multline}\label{qtheta}
\begin{split}
q_{\theta i} +\sum_{1\leq j \leq 3}2\rho_{\theta} U_{\theta j}\Theta_{\theta ij}+ \rho_{\theta} U_{\theta i}|U_{\theta}|^2 + \rho_{\theta} U_{\theta i} (\Theta_{\theta 11}+\Theta_{\theta 22}+\Theta_{\theta 33}) -5\rho_{\theta} U_{\theta i} \cr
= \theta  \left(q_i +\sum_{1\leq j \leq 3}2\rho U_j\Theta_{ij}+ \rho U_i|U|^2 + \rho U_i (\Theta_{11}+\Theta_{22}+\Theta_{33}) -5\rho U_i\right).
\end{split}
\end{multline}
We then apply (1)-(3) of this lemma and Lemma \ref{macro} to obtain the desired result.
\end{proof}

\begin{lemma}\label{macrotheta diff} Let $|\alpha| \geq 1$. For a sufficiently small $\mathcal{E}(t)$, there exist positive constants $C$ and $C_{\alpha}$ such that
\begin{align*}
&(1) \ |\partial^{\alpha}\rho_{\theta}(x,t)| \leq C\|\partial^{\alpha}f\|_{L^2_v},	\cr
&(2) \ |\partial^{\alpha}U_{\theta}(x,t)| \leq C_{\alpha}\sum_{|\alpha_1|\leq|\alpha|}\|\partial^{\alpha_1}f\|_{L^2_v},	\cr
&(3) \ |\partial^{\alpha}\Theta_{\theta ij}(x,t)| \leq C_{\alpha}\sum_{|\alpha_1|\leq|\alpha|}\|\partial^{\alpha_1}f\|_{L^2_v} ,	\cr
&(4) \ |\partial^{\alpha}q_{\theta i}(x,t)|  \leq C_{\alpha}\sum_{|\alpha_1|\leq|\alpha|}\|\partial^{\alpha_1}f\|_{L^2_v},
\end{align*}
for $1\leq i,j \leq 3$.
\end{lemma}
\begin{proof}
(1) The definition of $\rho_{\theta}$ and Lemma \ref{macro diff} yield
\begin{align*}
|\partial^{\alpha}\rho_{\theta}| = \theta|\partial^{\alpha}\rho| \leq C\|\partial^{\alpha}f\|_{L^2_v}.
\end{align*}
(2) Using the definition of $U_{\theta}$, we can write 
\begin{align*}
|\partial^{\alpha}U_{\theta}| =  \theta\bigg|\partial^{\alpha}\frac{ \rho U}{\rho_{\theta}} \bigg|.
\end{align*}
Since $\rho U \leq C\sqrt{\mathcal{E}(t)} $, by exactly the same argument in Lemma \ref{macro diff} (2), we obtain  
\begin{align*}
|\partial^{\alpha}U_{\theta}| \leq C_{\alpha}\sum_{|\alpha_1|\leq|\alpha|}\|\partial^{\alpha_1}f\|_{L^2_v}.
\end{align*}
(3) For a convenience of notation $\Theta_{\theta}$, we combine the previous computation \eqref{Thetaii} and \eqref{Thetaij} as follows: 
\begin{align*}
\Theta_{\theta ij}=\theta \frac{\rho\Theta_{ij}+\rho U_iU_j-\rho\delta_{ij}}{\rho_{\theta}} -U_{\theta i}U_{\theta j}+\delta_{ij}.
\end{align*}
Note that the numerator part can be estimated by Lemma \ref{macro} and Lemma \ref{macro diff} as
\begin{align*}
|\partial^{\alpha}\left( \rho\Theta_{ij}+\rho U_iU_j-\rho\delta_{ij}\right)| \leq C_{\alpha}\mathcal{E}(t)\sum_{|\alpha_1|\leq|\alpha|}\|\partial^{\alpha_1}f\|_{L^2_v}.
\end{align*}
Then from the same way as in Lemma \ref{macro diff} (2), we have 
\begin{align*}
|\partial^{\alpha}\Theta_{\theta ij}| \leq C_{\alpha}\sum_{|\alpha_1|\leq|\alpha|}\|\partial^{\alpha_1}f\|_{L^2_v}.
\end{align*}
(4) Recall that the form of $q_{\theta}$ in \eqref{qtheta}. We already obtained the estimates for all other terms $\rho$, $U$, $\Theta$. Therefore taking $\partial^{\alpha}$ on \eqref{qtheta} gives 
\begin{align*}
|\partial^{\alpha}q_{\theta i}|  \leq C_{\alpha}\sum_{|\alpha_1|\leq|\alpha|}\|\partial^{\alpha_1}f\|_{L^2_v}.
\end{align*}
\end{proof}
\subsection{Estimate for the nonlinear term}
We now estimate the nonlinear perturbations.
\begin{proposition}\label{prop} Let $\mathcal{E}(t)$ be sufficiently small. Then we have 
\begin{align*}
\bigg| \int_{\mathbb{R}^3}\partial^{\alpha}_{\beta} \Gamma(f) g dv \bigg|	 &\leq  C\sqrt{\mathcal{E}(t)}\sum_{|\alpha_1|+|\alpha_2|+|\alpha_3|\leq |\alpha|} \|\partial^{\alpha_1}f\|_{L^2_{v}}\|\partial^{\alpha_2}f\|_{L^2_{v}}\|\partial^{\alpha_3}f\|_{L^2_v}\|g\|_{L^2_{v}}.
\end{align*}
\end{proposition}
\begin{proof} Since the other terms are similar, we only consider $\Gamma_2$. We apply $\partial^{\alpha}_{\beta}$ to the non-linear term $\Gamma_2$ in \eqref{non linear}:
\begin{multline*}
\partial^{\alpha}_{\beta}\Gamma_2(f) =\sum_{\substack{1 \leq i,j \leq 13\cr \alpha_1+\alpha_2+\alpha_3=\alpha}} \frac{1}{\tau_0}\int_0^1 \partial^{\alpha_1}_{\beta}\left(\frac{1}{\sqrt{m}}\left\{D^2_{(\rho_{\theta},\rho_{\theta} U_{\theta},G_{\theta},H_{\theta})}\mathcal{S}_{Pr}(\theta)\right\}_{ij}\right)(1-\theta)d\theta\cr
\times  \langle \partial^{\alpha_2}f, e_i \rangle_{L^2_v}\langle \partial^{\alpha_3}f, e_j \rangle_{L^2_v}.
\end{multline*}
Since we have from Lemma \ref{non-linear poly} that
\begin{align*}
	\{D^2_{(\rho_{\theta},\rho_{\theta} U_{\theta},G_{\theta},H_{\theta})}\mathcal{S}_{Pr}(\theta)\}_{ij} =\frac{\mathcal{P}_{ij}(\rho_{\theta},U_{\theta},\Theta_{\theta},q_{\theta},(v_i-U_{\theta i}),1-Pr)}{\rho_{\theta}^{m}T_{\theta}^{n}}\mathcal{M}(\theta).
\end{align*}
We consider  
\begin{align*}
\partial^{\alpha}_{\beta}\left(\frac{\mathcal{P}_{ij}(\rho_{\theta},U_{\theta},\Theta_{\theta},q_{\theta},(v_i-U_{\theta i}),1-Pr)}{\rho_{\theta}^{m}T_{\theta}^{n}}\mathcal{M}(\theta)/\sqrt{m}\right).
\end{align*}
Using the estimates of macroscopic fields in Lemma \ref{macrotheta} and Lemma \ref{macrotheta diff}, we have 
\begin{align*}
\partial^{\alpha}_{\beta}\left(\frac{\mathcal{P}_{ij}(\rho_{\theta},U_{\theta},\Theta_{\theta},q_{\theta},(v_i-U_{\theta i}),1-Pr)}{\rho_{\theta}^{m}T_{\theta}^{n}}\right) \leq C_{\alpha}\sum_{|\alpha_1|\leq|\alpha|}\|\partial^{\alpha_1}f\|_{L^2_v}\mathcal{P}(v_i),
\end{align*}
for some generically defined polynomial $\mathcal{P}$. The remaining exponential part can be estimated similarly:
\begin{align*}
\partial^{\alpha}_{\beta}\left(\mathcal{M}(\theta)/\sqrt{m}\right) &\leq C_{\alpha}\sum_{|\alpha_1|\leq|\alpha|}\|\partial^{\alpha_1}f\|_{L^2_v}\mathcal{P}(v_i) \exp\left(-\frac{|v-U_{\theta}|^2}{2T_{\theta}}+\frac{|v|^2}{4}\right).
\end{align*}
Combining these computations with the H\"{o}lder inequality yields 
\begin{multline*}
\bigg| \int_{\mathbb{R}^3}\partial^{\alpha}_{\beta} \Gamma_2(f) g dv \bigg| \leq  C\sum_{|\alpha_1|+|\alpha_2|+|\alpha_3|\leq |\alpha|} \|\partial^{\alpha_1}f\|_{L^2_{v}}\|\partial^{\alpha_2}f\|_{L^2_{v}}\|\partial^{\alpha_3}f\|_{L^2_v}\|g\|_{L^2_{v}} \cr
\times  \left(\int_{\mathbb{R}^3}\mathcal{P}(v_i)\exp\left(-\frac{|v-U_{\theta}|^2}{2T_{\theta}}+\frac{|v|^2}{2}\right)dv\right)^{\frac{1}{2}} .
\end{multline*}
Then, since $U_{\theta}<1$ and $T_{\theta}\leq 3/2$ for sufficiently small $\mathcal{E}(t)$, we have  
\begin{align*}
\exp\left(-\frac{|v-U_{\theta}|^2}{2T_{\theta}}+\frac{|v|^2}{2}\right) \leq \exp\left(-\frac{|v-4U_{\theta}|^2}{6}+2|U_{\theta}|^2\right)\leq C\exp\left(-\frac{|v|^2}{6}\right).
\end{align*}
%then we can conclude that
%\begin{align*}
%\int_{\mathbb{R}^3}\mathcal{P}(v_i)\exp\left(-\frac{|v-4U_{\theta}|^2}{6}+2|U_{\theta}|^2\right)dv&=\int_{\mathbb{R}^3}\mathcal{P}(v_i+4U_{\theta i})\exp\left(-|v|^2/6\right)dv \exp(2|U_{\theta}|^2) \cr
%&\leq  C\sqrt{\mathcal{E}(t)}.
%\end{align*}
This completes the proof.
\end{proof}
\subsection{Local solution}
We are now ready to construct the local smooth solution.
\begin{theorem}\label{local} Let $N\geq 3$ and $F_0(x,v)\geq 0 $. Then there exists $M_0$ and $T_*\geq0 $ such that if $\mathcal{E}(f_0) \leq M_0/2$, then \eqref{perturb} has the unique local-in-time classical solution that exists for $0\leq t<T_*$ satisfying 
\begin{enumerate}
\item The energy of perturbation is continuous and satisfies
\begin{align*}
	\sup_{0 \leq t \leq T_*}\mathcal{E}(f)(t) \leq M_0.
\end{align*}
\item The distribution function is non-negative: 
\begin{align*}
	F(x,v,t)=m+\sqrt{m}f(x,v,t) \geq 0.
\end{align*}
\item The Shakhov operator is non-negative: 
\[\mathcal{S}_{Pr}(F)\geq 0. \]
\item The perturbation $f$ satisfies the conservation laws \eqref{conservf} for $ 0\leq t \leq T_*$.

\end{enumerate}
\end{theorem}
\begin{proof}
We define $F^{n+1}$ iteratively by the following scheme:
\begin{align}\label{Iter}
\begin{split}
\partial_tF^{n+1}+v\cdot \nabla_xF^{n+1} &= \frac{1}{\tau(F^n)}(\mathcal{S}_{Pr}(F^n)-F^{n+1}), \cr
F^{n+1}(x,v,0) &= F_0(x,v),
\end{split}
\end{align}
with $F^0(x,v,t)=F_0(x,v)$.  We use induction argument. We assume the statement (1) - (4) for $n$-th step.  
%\begin{align}\label{induction}
%F^n(x,v,t) \geq 0 , \qquad \mathcal{E}(f^n)(t) \leq M_0.\quad [0,T_*]
%\end{align}
%for some $T_*$.
We first observe that
\begin{align*}
	\mathcal{S}_{Pr}(F^n) &= \frac{\rho_n}{\sqrt{2\pi T_n}^3}  \left[ 1+\frac{1-Pr}{5}\frac{q_n\cdot v}{\rho_n T_n^2} \left(\frac{|v|^2}{2T_n}-\frac{5}{2}\right)\exp\left(-\frac{|v|^2}{2T_n}\right)\right]\cr
	%&= \frac{\rho_n}{\sqrt{2\pi T_n}^3}  \left[ 1+\frac{1-Pr}{5}\frac{q_n\cdot v}{\rho_n T_n^2} \left(\frac{|v|^2}{2T_n}-\frac{5}{2}\right)\exp\left(-\frac{|v|^2}{2T_n}\right)\right]\cr
	&\geq \mathcal{M}(F^n)\left(1-C\frac{1-Pr}{5}\frac{|q_{n}|}{\rho_{n} T_{n}^{3/2}}\right),
\end{align*}
%where $\rho_n$, $T_n$, $q_n$ are macroscopic quantities constructed from $F^n$ by \eqref{macro quantity}, and 
where we used
\begin{align*}
	\bigg|x^m\exp\left(-\frac{x^2}{2T}\right)\bigg| \leq C T^{m/2}.
\end{align*}
%to get
%\begin{align*}
%	\mathcal{S}_{Pr}(F^n)\geq \mathcal{M}(F^n)\left(1-C\frac{1-Pr}{5}\frac{|q_{n}|}{\rho_{n} T_{n}^2}\right).
%\end{align*}
Applying Lemma \ref{macro} and the induction hypothesis yields  
\begin{align*}
	\mathcal{S}_{Pr}(F^n)\geq \mathcal{M}(F^n)\left(1-C\sqrt{\mathcal{E}(f^n)}\right)\geq \mathcal{M}(F^n)\left(1-CM_0\right).
\end{align*}
Therefore for sufficiently small $M_0$, we have $\mathcal{S}_{Pr}(F^n)\geq0$. Then the non-negativity of $F^n$ follows directly from the mild formulation of \eqref{Iter}:
\begin{multline*}
F^{n+1}(x,v,t) = e^{-\int_0^t\frac{1}{\tau(F^n)}dt}F_0(x-vt,v)\cr
+\frac{1}{\tau(F^n)}e^{-\int_s^t\frac{1}{\tau(F^n)}dt}\int_0^{T_*}\mathcal{S}_{Pr}(F^n)(x+v(s-t),v,s)ds.
\end{multline*}
%the non-negativity of $\mathcal{S}_{Pr}(F^n)$ implies the non-negativity of the next iteration solution $F^{n+1}$.
Now we prove the uniform boundedness of the energy norm. For this, we substitute $F^{n+1}=m+\sqrt{m}f^{n+1}$ into \eqref{Iter} to get
\begin{align*}
\partial_tf^{n+1}+v\cdot \nabla_xf^{n+1} +\frac{1}{\tau_0}f^{n+1} &= \frac{1}{\tau_0}P_{Pr}(f^n)+\Gamma(f^n), \cr
f^{n+1}(x,v,0) &= f_0(x,v),
\end{align*}
where $f_0^n(x,v)=(F_0(x,v)-m)/\sqrt{m}$. Taking $\partial^{\alpha}_{\beta}$ on both sides: 
\begin{align*}
\partial_t\partial^{\alpha}_{\beta} f^{n+1}+v\cdot \nabla_x \partial^{\alpha}_{\beta} f^{n+1} +\frac{1}{\tau_0}\partial^{\alpha}_{\beta} f^{n+1}+ \sum_{i=1}^3 \partial^{\alpha+\bar{k}_i}_{\beta-k_i}\partial^{\alpha}_{\beta} f^{n+1}&= \frac{1}{\tau_0}\partial_{\beta}P_{Pr}(\partial^{\alpha}f^n)+\partial^{\alpha}_{\beta}\Gamma(f^n),
\end{align*}
and taking  product with $ \partial^{\alpha}_{\beta} f^{n+1}$ yields
\begin{multline*}
\frac{1}{2}\frac{d}{dt}\|\partial^{\alpha}_{\beta} f^{n+1}\|_{L^2_{x,v}}^2+\frac{1}{\tau_0}\|\partial^{\alpha}_{\beta} f^{n+1}\|_{L^2_{x,v}}^2 \leq  \sum_{i=1}^3 \int_{\mathbb{R}^3\times \mathbb{T}^3}\partial^{\alpha}_{\beta} f^{n+1}\partial^{\alpha+\bar{k}_i}_{\beta-k_i}\partial^{\alpha}_{\beta} f^{n+1}dvdx \cr
+ \frac{1}{\tau_0}\int_{\mathbb{R}^3\times \mathbb{T}^3}(\partial^{\alpha}_{\beta} f^{n+1}\partial_{\beta}P_{Pr}(\partial^{\alpha}f^n)+\partial^{\alpha}_{\beta} f^{n+1}\partial^{\alpha}_{\beta}\Gamma(f^n))dvdx,
\end{multline*}
where $k_i~(i=1,2,3)$ are coordinate unit vectors  and $\bar{k}_1=(0,1,0,0)$, $\bar{k}_2=(0,0,1,0)$, $\bar{k}_3=(0,0,0,1)$.
We then integrate for time and recall Proposition \ref{prop} to derive
\begin{align*}
(1-CT_*)\mathcal{E}(f^{n+1})(t) \leq \left( \frac{1}{2}+CT_*+CT_*M_0^2\right)M_0.
\end{align*}
For sufficiently small $M_0$ and $T_*$, we conclude that
\begin{align*}
\mathcal{E}(f^{n+1})(t) \leq M_0.
\end{align*}
The remaining part can be obtained from the standard argument \cite{MR1908664,MR2000470,MR2095473}. We omit it. 
\end{proof}

\section{Coercivity estimate for $Pr>0$}

In this section, we fill up the degeneracy of the linearized Shakov operator $L_{Pr}$ and recover the full coercivity. As is observed in Proposition \ref{dichotomy} and Lemma \ref{ker}, the degeneracy of the linear operator $L_{Pr}$ is strictly larger when  $Pr=0$ than $Pr>0$. The case of $Pr >0$ can be treated by a rather standard argument, which is briefly presented in this section. 

Recall the macroscopic projection operator $P_c$ from Definition \ref{P_{Pr}f}:
\begin{align*}
P_cf=a_0(x,t)\sqrt{m}+\sum_{1\leq i\leq 3}b_{0i}(x,t)v_i\sqrt{m}+c_0(x,t)|v|^2\sqrt{m},
\end{align*}
where
\begin{align*}
&a_0(x,t)=\int_{\mathbb{R}^3} f\sqrt{m}dv-\frac{1}{2}\int_{\mathbb{R}^3} f(|v|^2-3)\sqrt{m}dv ,\cr
&b_{0i}(x,t)=\int_{\mathbb{R}^3} fv_i\sqrt{m}dv, \cr
&c_0(x,t)=\frac{1}{6}\int_{\mathbb{R}^3} f(|v|^2-3)\sqrt{m}dv,
\end{align*}
for $i=1,2,3$. Substituting $f=P_cf+(I-P_c)f$ into \eqref{perturb} gives 
\begin{align}\label{split0}
	\{\partial_t+v\cdot \nabla_x \}  \{ P_cf \} = l_0(f)+h_0(f),
\end{align}
where
\begin{align*}
	l_0(f)&= -\{\partial_t+v\cdot \nabla_x\}\{(I-P_c)f\}+\frac{1}{\tau_0}L_{Pr}\{(I-P_c)f\}, \quad 
	h_0(f)=\Gamma(f).
\end{align*}
By an explicit computation, the left-hand side of \eqref{split0} is expressed as a linear combination of the following $13$-basis 
$	\{\sqrt{m},v_i\sqrt{m},v_iv_j\sqrt{m},v_i|v|^2\sqrt{m}\}$ $(1\leq i,j\leq 3)$:
\begin{multline*}
\bigg\{\partial_ta_0+\sum_{1\leq i\leq 3}(\partial_{x_i}a_0+\partial_tb_{0i})v_i  +\sum_{1\leq i\leq 3}(\partial_{x_i}b_{0i}+\partial_tc_0)v_i^2\cr
+\sum_{i<j}(\partial_{x_i}b_{0j}+\partial_{x_j}b_{0i})v_iv_j 
+\sum_{1\leq i\leq 3}(\partial_{x_i}c_0)v_i|v|^2\bigg\}\sqrt{m}.
\end{multline*}
Let $(l_{0c}, l_{0i}, l_{0ij}, l_{0is}, l_{0ijs})$, and  $(h_{0c}, h_{0i}, h_{0ij}, h_{0is},  h_{0ijs})$ be the coefficient corresponding to the linear 
expansion w.r.t the above basis when $l_0$ and $h_0$ are expanded to the above $13$ basis, respectively. Then comparing the coefficients of both sides yields the following system:
\begin{align}\label{system1}
\begin{split}
	\partial_ta &= l_{0c}+h_{0c}, \cr
	\partial_{x_i}a+\partial_tb_i  &= l_{0i}+h_{0i},	\cr
	\partial_{x_i}b_i+ \partial_{t}c  &= l_{0ii} + h_{0ii},	\cr
	\partial_{x_i}b_j + \partial_{x_j}b_i &= l_{0ij}+h_{0ij} \quad (i\neq j),	\cr
	\partial_{x_i}c &= l_{0is} + h_{0is}.
\end{split}
\end{align}
for $i,j=1,2,3$. For the notational simplicity we define 
\begin{align*}
	\tilde{l}_0 &=l_{0c}+\sum_{1\leq i \leq 3}\left(l_{0i}+l_{0is}\right)+\sum_{1\leq i,j \leq 3}\left(l_{0ij}+ l_{0ijs}\right), \cr
	\tilde{h}_0 &= h_c+\sum_{1\leq i \leq 3}\left(h_0{i}+h_{0is}\right)+\sum_{1\leq i,j \leq 3}\left(h_{0ij}+ h_{0ijs}\right).
\end{align*}
The analysis for this system is now standard, which can be found, for example, in \cite{MR1908664,MR2000470,MR2095473,MR2779616} to yield
\begin{align*}
\sum_{|\alpha|\leq N}\|P_{Pr}\partial^{\alpha}f\|_{L^2_{x,v}} &\leq C\sum_{|\alpha|\leq N} \left(\|\partial^{\alpha}a\|_{L^2_x}+\|\partial^{\alpha}b\|_{L^2_x}+\|\partial^{\alpha}c\|_{L^2_x}\right) \cr
&\leq C \sum_{|\alpha|\leq N-1}\left( \|\partial^{\alpha}\tilde{l}_0\|_{L_{x}^2} + \|\partial^{\alpha}\tilde{h}_0\|_{L_{x}^2}\right) \cr
&\leq  C \sum_{|\alpha|\leq N}\|(I-P_{Pr})\partial^{\alpha}f\|_{L^2_{x,v}} +CM_0 \sum_{|\alpha|\leq N}\|\partial^{\alpha}f\|_{L^2_{x,v}},
\end{align*}
for sufficiently small $\mathcal{E}(t)$. This, combined with the degenerate coercive estimate in Proposition \ref{dichotomy} (1), leads to the following full coercivity estimate for sufficiently small $\mathcal{E}(t)$: 
\begin{align}\label{full coer0}
\sum_{|\alpha|\leq N}\langle L_{Pr}\partial^{\alpha}f, \partial^{\alpha}f\rangle_{L^2_{x,v}} \leq -\delta \sum_{|\alpha|\leq N}\|\partial^{\alpha}f\|_{L^2_{x,v}}^2.
\end{align}
%%%%%%%%%%%%%%%%%%%%%%%%%%%%%%%%%%%%%%%%%%%%%%%%%%%%%%%%%%%%%%%%%%%%%%%%%%%%%%%%%%%%%%%%%%%%%%%%%%%%%%%%%%%%%%%%%%%%%%%%%%%%%%%%%%%%%%%%%%%%%%%%%%%%%%%%%%%%%%%%%
%
%         PR=0
%
%
%%%%%%%%%%%%%%%%%%%%%%%%%%%%%%%%%%%%%%%%%%%%%%%%%%%%%%%%%%%%%%%%%%%%%%%%%%%%%%%%%%%%%%%%%%%%%%%%%%%%%%%%%%%%%%%%%%%%%%%%%%%%%%%%%%%%%%%%%%%%%%%%%%%%%%%%%%%%%%%%
\section{Coercivity estiamte for $Pr=0$}
Due to the bigger degeneracy of $L_{Pr}$ in the case of $Pr=0$, a new idea is to need to recover the full coercivity. More precisely, the presence of an additional 3-dimensional null space leads to a bigger differential system than \eqref{system1} that involves non-conservative quantities, which cannot be merged into the coercivity estimate using the previous arguments.  

First,  we write $P_{Pr}f$ in Definition \ref{P_{Pr}f} as
\begin{align*}
P_{Pr}f=a(x,t)\sqrt{m}+\sum_{1\leq i\leq 3}b_i(x,t)v_i\sqrt{m}+c(x,t)|v|^2\sqrt{m}+\sum_{1\leq i\leq 3}d_i(x,t)v_i|v|^2\sqrt{m},
\end{align*}
where
\begin{align*}
&a(x,t)=\int_{\mathbb{R}^3} f\sqrt{m}dv-\frac{1}{2}\int_{\mathbb{R}^3} f(|v|^2-3)\sqrt{m}dv ,\cr
&b_i(x,t)=\int_{\mathbb{R}^3} fv_i\sqrt{m}dv-\frac{1}{2}\int_{\mathbb{R}^3} fv_i(|v|^2-5)\sqrt{m}dv, \cr
&c(x,t)=\frac{1}{6}\int_{\mathbb{R}^3} f(|v|^2-3)\sqrt{m}dv, \cr
&d_i(x,t)=\frac{1}{10}\int_{\mathbb{R}^3} fv_i(|v|^2-5) \sqrt{m}dv,
\end{align*}
for $i=1,2,3$. Substituting $f=P_{Pr}f+(I-P_{Pr})f$ into \eqref{perturb}, we get the similar equation as in \eqref{split0}:
\begin{align}\label{split}
\{\partial_t+v\cdot \nabla_x \}  \{ P_{Pr}f \} = l(f)+h(f),
\end{align}
where
\begin{align*}
l(f)&= -\{\partial_t+v\cdot \nabla_x\}\{(I-P_{Pr})f\}+\frac{1}{\tau_0}L_{Pr}\{(I-P_{Pr})f\}, \quad 
h(f)=\Gamma(f).
\end{align*}
This time, we expand the L.H.S. of \eqref{split} using the following $19$-basis $(1\leq i,j\leq 3)$:
\begin{align}\label{basis}
\{\sqrt{m},v_i\sqrt{m},v_iv_j\sqrt{m},v_i|v|^2\sqrt{m},v_iv_j|v|^2\sqrt{m}\},
\end{align}
to get
\begin{multline*}
\bigg\{\partial_ta+\sum_{1\leq i\leq 3}(\partial_{x_i}a+\partial_tb_i)v_i  +\sum_{1\leq i\leq 3}(\partial_{x_i}b_i+\partial_tc)v_i^2+\sum_{i<j}(\partial_{x_i}b_j+\partial_{x_j}b_i)v_iv_j \cr
+\sum_{1\leq i\leq 3}(\partial_{x_i}c+\partial_td_i)v_i|v|^2+\sum_{1\leq i\leq 3}\partial_{x_i}d_iv_i^2|v|^2+\sum_{i<j}(\partial_{x_i}d_j+\partial_{x_j}d_i)v_iv_j|v|^2\bigg\}\sqrt{m}.
\end{multline*}
Let $(l_c, l_i, l_{ij}, l_{is}, l_{ijs})$, and  $(h_c, h_i, h_{ij}, h_{is},  h_{ijs})$ be the coefficient for the linear expansion of  $l$ and $h$ w.r.t \eqref{basis} respectively. Then we can derive the following system:
\begin{align}\label{system}
\begin{split}
\partial_ta &= l_c+h_c, \cr
\partial_{x_i}a+\partial_tb_i  &= l_{i}+h_{i},	\cr
\partial_{x_i}b_i+\partial_{t}c  &= l_{ii} + h_{ii},	\cr
\partial_{x_i}b_j + \partial_{x_j}b_i &= l_{ij}+h_{ij} \quad (i\neq j),	\cr
\partial_{x_i}c+\partial_td_i &= l_{is} + h_{is}, \cr
\partial_{x_i}d_i &= l_{iis} + h_{iis}, \cr
\partial_{x_i}d_j+\partial_{x_j}d_i &= l_{ijs} + h_{ijs} \quad (i\neq j),
\end{split}
\end{align}
for $i,j=1,2,3$. For a notational simplicity we define 
\begin{align*}
\tilde{l} &=l_c+\sum_{1\leq i \leq 3}\left(l_{i}+l_{is}\right)+\sum_{1\leq i,j \leq 3}\left(l_{ij}+ l_{ijs}\right), \cr
\tilde{h} &= h_c+\sum_{1\leq i \leq 3}\left(h_{i}+h_{is}\right)+\sum_{1\leq i,j \leq 3}\left(h_{ij}+ h_{ijs}\right).
\end{align*}
The desired full coercivity estimate is stated in the following theorem. Note that we need an additional moment condition on the initial data. 

\begin{theorem}\label{full coer}
	Let $Pr=0$ and $|\alpha|\leq N$. Let $f$ be the local smooth solution obtained in Theorem \ref{local}. Suppose further that the third moment of the initial data is zero: 
	\begin{align}\label{assum}
	\int_{\mathbb{T}^3\times\mathbb{R}^3} F_0(x,v)v_i|v|^2dvdx= 0,
	\end{align}
	for $i=1,2,3$. Then we have
	\begin{align*}
	\sum_{|\alpha|\leq N}\langle L_{Pr}\partial^{\alpha}f, \partial^{\alpha}f\rangle_{L^2_{x,v}} \leq -\delta \sum_{|\alpha|\leq N}\|\partial^{\alpha}f\|_{L^2_{x,v}}^2+C\mathcal{E}^2(t).
	\end{align*}
\end{theorem}
The following estimate of macroscopic variables is the key estimate for the proof of Theorem \ref{full coer}.
\begin{proposition}\label{abcd} Under the same assumption as in Theorem \ref{full coer}, we have
\begin{align*}
\|\partial^{\alpha}a\|_{L_{x}^2}+\|\partial^{\alpha}b\|_{L_{x}^2}+\|\partial^{\alpha}c\|_{L_{x}^2}+\|\partial^{\alpha}d\|_{L_{x}^2} 
\leq C \sum_{|\alpha|\leq N-1}\left( \|\partial^{\alpha}\tilde{l}\|_{L_{x}^2} + \|\partial^{\alpha}\tilde{h}\|_{L_{x}^2}\right)+C\mathcal{E}(t).
\end{align*}
\end{proposition}
\begin{proof}
Note that the new variable $d$ is coupled only with $c$. Therefore, the estimates for $a$ and $b$ are the same as the previous $Pr>0$ case, which is
\begin{align}\label{ab}
\|\partial^{\alpha}a\|_{L_{x}^2}+\|\partial^{\alpha}b\|_{L_{x}^2}
\leq C \sum_{|\alpha|\leq N-1}\left( \|\partial^{\alpha}\tilde{l}\|_{L_{x}^2} + \|\partial^{\alpha}\tilde{h}\|_{L_{x}^2}\right).
\end{align}
We divide the estimate of $c$ and $d$ into the following four steps:
\begin{align*}
	&(\mbox{Step} ~ 1) \quad  \|\partial^{\alpha}\partial_tc\|_{L_{x}^2},	\cr
	&(\mbox{Step} ~ 2) \quad	\|\nabla_x\partial^{\alpha}d_i\|_{L_{x}^2}, \quad (|\alpha|\leq N-1)\cr
	&(\mbox{Step} ~ 3) \quad \|\partial^{\gamma}d_i\|_{L_{x}^2},  \cr
	&(\mbox{Step} ~ 4) \quad \|c\|_{L_{x}^2}+\|\nabla_x\partial^{\alpha}c\|_{L_{x}^2},
\end{align*}
where $\partial^{\gamma}$ is pure time derivatives for $|\gamma|\leq N$, that is, $\partial^{\gamma}=\partial_t^{|\gamma|}$. \newline

\noindent$\bullet$ {\bf Step 1.} The estimate of $\|\partial^{\alpha}\partial_tc\|_{L_{x}^2}$: Taking $\partial^{\alpha}$ on $\eqref{system}_3$ gives
\begin{align*}
\partial^{\alpha}\partial_{t}c &= \partial^{\alpha}l_{ii} + \partial^{\alpha}h_{ii}-\partial^{\alpha}\partial_{x_i}b_i.
\end{align*}
Multiplying both sides by $\partial^{\alpha}\partial_{t}c$  and applying the H\"{o}lder inequality yields 
\begin{align*}
\|\partial^{\alpha}\partial_{t}c\|_{L^2_x} &\leq \|\partial^{\alpha}l_{ii}\|_{L^2_x} + \|\partial^{\alpha}h_{ii}\|_{L^2_x}+\|\partial^{\alpha}\partial_{x_i}b_i\|_{L^2_x}.
\end{align*}
We then combine this with the estimate of $b$ in \eqref{ab} to get
\begin{align*}
\|\partial^{\alpha}\partial_{t}c\|_{L^2_x} &\leq C\sum_{|\alpha|\leq N-1} \left(\|\partial^{\alpha}\tilde{l}\|_{L^2_x}+\|\partial^{\alpha}\tilde{h}\|_{L^2_x}\right).
\end{align*}
%%%%%%%%%%%%%%%%%%%%%%%%%%%%%%%%%%%%%%%%%%%%%%%%%%%%%%%%%%%%%%%%%%%%%%%%%%%%%%%%%%%%%%%%%%%%%%%%%%%%%%%%%%%%%%%%%%%%%%%%%%%%%%%%%%%%%%%%%%%
\noindent$\bullet$ {\bf Step 2.} The estimate of $\|\nabla_x\partial^{\alpha}d_i\|_{L_{x}^2}$: Using $\eqref{system}_6$ and $\eqref{system}_7$, we compute
\begin{align*}
	\triangle d_i 
	&= \sum_{1\leq j\leq 3}\partial_{jj}d_i\\ 
	&= \sum_{j \neq i} \partial_{jj} d_i +\partial_{ii} d_i \\
	&= \sum_{j \neq i} \left(\partial_{j}l_{ijs} + \partial_{j}h_{ijs}-\partial_{ji} d_j\right) + \partial_{i}l_{iis} + \partial_{i}h_{iis}.
\end{align*}
We then use $\eqref{system}_6$ to get
\begin{align*}
	\triangle d_i &= \sum_{j \neq i} \left(\partial_{j}l_{ijs} + \partial_{j}h_{ijs}-\partial_{i}l_{jjs} -\partial_{i}h_{jjs}\right) + \partial_{i}l_{iis} + \partial_{i}h_{iis},
\end{align*}
which implies
\begin{align*}
\|\nabla_x d\|_{L^2_x} \leq C \sum_{1 \leq i,j \leq 3} (\|l_{ijs}\|_{L^2_x}+\|h_{ijs}\|_{L^2_x}+\|l_{iis}\|_{L^2_x}+\|h_{iis}\|_{L^2_x}).
\end{align*}
The same argument holds when  $d_i$ is replaced by $\partial^{\alpha}d_i$. This completes the proof of Step 2. \newline 

%%%%%%%%%%%%%%%%%%%%%%%%%%%%%%%%%%%%%%%%%%%%%%%%%%%%%%%%%%%%%%%%%%%%%%%%%%%%%%%%%%%%%%%%%%%%%%%%%%%%%%%%%%%%%%%%%%%%%%%%%%%%%%%%%%%%%%%%%%%
\noindent$\bullet$ {\bf Step 3.} The estimate of $\|\partial^{\gamma}d_i\|_{L_{x}^2}$: We divide the proof into the following 3 cases: $|\gamma|=0$ $|\gamma|=1$ and $2\leq |\gamma|\leq N$. We start with $|\gamma|=1$.\newline

\noindent (1) The case of $|\gamma|=1$: We employ the  Poincar\'{e} inequality to derive
\begin{align}\label{poin d}
	\|\partial_t d_i\|_{L^2_x} &\leq \|\nabla_x \partial_t d_i\|_{L^2_x}+C\Big\|\int_{\mathbb{T}^3}\partial_t d_i dx\Big\|_{L^2_x}.
\end{align}
Note that, in the case of $a$, $b$, $c$, the last term on the R.H.S. vanishes due to the conservation laws, which is not the case for non-conservative quantity $d$.
To control $\int \partial^{\gamma}d_i dx$, we multiply $v_i|v|^2$ and integrate with respect to $dvdx$ on the equation \eqref{Shakhov}.
\begin{align}\label{b0}
\frac{d}{dt}\int_{\mathbb{T}^3\times\mathbb{R}^3}Fv_i|v|^2 dvdx &= \frac{1}{\tau}\int_{\mathbb{T}^3\times\mathbb{R}^3}(\mathcal{S}_{Pr}(F)-F)v_i|v|^2~ dvdx.
\end{align}
We recall from $\eqref{viv2}_2$ that the energy flux can be expressed as follows:
\begin{align}\label{b1}
\int_{\mathbb{R}^3}F(x,v,t)v_i|v|^2dv &=q_i +\sum_{1\leq j \leq 3}2\rho U_j\Theta_{ij}+ \rho U_i|U|^2 + U_i \rho  \sum_{1\leq i \leq 3 }\Theta_{ii}.
\end{align}
On the other hand, an explicit computation using the following decomposition of $v_i|v|^2$
\begin{multline*}
v_i|v|^2 =(v_i-U_i)|v-U|^2 +2(v_i-U_i)(v-U)\cdot U +(v_i-U_i)|U|^2 \cr
+U_i|v-U|^2 + 2U_i(v-U)\cdot U + U_i|U|^2,
\end{multline*}
gives
\begin{align}\label{b2}
\int_{\mathbb{R}^3}\mathcal{S}_{Pr}(F)(x,v,t)v_i|v|^2dv&=(1-Pr)q_i +2U_i\rho T + \rho U_i|U|^2 + U_i 3\rho  T.
\end{align}
Inserting \eqref{b1} and \eqref{b2} into \eqref{b0}, we get the following evolution law for the energy flux for $Pr=0$: 
%\begin{align*}
%\int_{\mathbb{R}^3}(\mathcal{S}_{Pr}(F)-F)v_i|v|^2 dv = 2U_i\rho T - \sum_{1\leq j \leq 3}2\rho U_j\Theta_{ij},
%\end{align*}
%Thus we have
\begin{align}\label{intv|v|^2}
\frac{d}{dt}\int_{\mathbb{T}^3\times\mathbb{R}^3}F(x,v,t)v_i|v|^2dvdx= \frac{1}{\tau}\int_{\mathbb{T}^3}\bigg(2U_i\rho T - \sum_{1\leq j \leq 3}2\rho U_j\Theta_{ij}\bigg)dx,
\end{align}
%where we used $3T=\Theta_{11}+\Theta_{22}+\Theta_{33}$.
which, combined with the momentum conservation law:
\begin{align*}
\frac{d}{dt}\int_{\mathbb{T}^3\times\mathbb{R}^3}F(x,v,t)v~dvdx= 0,
\end{align*}
and
\[
\int_{\mathbb{R}^3}v_i|v|^2mdv= 0 ,
\]
%we obtain\begin{align*}
%\frac{1}{10}\frac{d}{dt}\int_{\mathbb{T}^3\times\mathbb{R}^3}Fv_i(|v|^2-5)dvdx&=\frac{1}{10}\frac{d}{dt}\int_{\mathbb{T}^3\times\mathbb{R}^3}fv_i(|v|^2-5)\sqrt{m}dvdx\cr
%&= \frac{1}{10\tau}\int_{\mathbb{T}^3}\bigg(2U_i\rho T - \sum_{1\leq j \leq 3}2\rho U_j\Theta_{ij}\bigg)dx,
%\end{align*}
gives the evolution law for $d_i$:
\begin{align}\label{dd}
\frac{d}{dt}\int_{\mathbb{T}^3}d_i(x,t)dx= \frac{1}{10\tau} \bigg( \int_{\mathbb{T}^3}2U_i\rho T - \sum_{1\leq j \leq 3}2\rho U_j\Theta_{ij}dx\bigg).
\end{align} 
Then, from Lemma \ref{macro},
\[U,~ \rho T,~ \rho \Theta \leq \| f \|_{L^2_v},\] 
we get
\begin{align}\label{d1}
\bigg|\int_{\mathbb{T}^3}\partial_td_i(x,t)dx\bigg|=\bigg|\frac{d}{dt}\int_{\mathbb{T}^3}d_i(x,t)dx\bigg|\leq C\int_{\mathbb{T}^3} \| f \|_{L^2_v}^2dx \leq C\mathcal{E}(t).
\end{align}
Inserting this, into the Poincar\'{e} inequality \eqref{poin d}, we get the desired estimate for $|\gamma|=1$:
\begin{align*}
\|\partial_t d_i\|_{L^2_x} &\leq \|\nabla_x \partial_t d_i\|_{L^2_x}+C\mathcal{E}(t).
\end{align*}
\newline
\noindent (2) The case of $|\gamma|=0$: 
We note from  \eqref{assum} that 
\begin{align*}
\int_{\mathbb{T}^3}d_i(x,0)dx %=\int_{\mathbb{T}^3\times\mathbb{R}^3} \frac{v_i(|v|^2-5)}{10}f(x,v,0)\sqrt{m}~dv dx
=\int_{\mathbb{T}^3\times\mathbb{R}^3} \frac{v_i(|v|^2-5)}{10}F(x,v,0)~dvdx = 0.
\end{align*}
Therefore, integration \eqref{dd} with respect to $dt$ gives
\begin{align}\label{d-d0}
\begin{split}
\int_{\mathbb{T}^3}d_i(x,t)dx = \frac{1}{10\tau}\int_0^t \int_{\mathbb{T}^3}\frac{2}{3}U_i\rho (\Theta_{11}+\Theta_{22}+\Theta_{33}-3\Theta_{ii}) - \sum_{j \neq i }2\rho U_j\Theta_{ij}dxdt,
\end{split}
\end{align}
where we used $3T=\Theta_{11}+\Theta_{22}+\Theta_{33}$.
%\begin{align*}
%\int_{\mathbb{T}^3}d_i(x,t)dx-\int_{\mathbb{T}^3}d_i(x,0)dx &= \frac{1}{10\tau}\int_0^t \int_{\mathbb{T}^3}\frac{2}{3}\frac{\rho U_i}{\rho}\rho \big(\sum_{j \neq i }\Theta_{jj}-\Theta_{ii}\big) - \sum_{j \neq i }2\frac{\rho U_j}{\rho}\rho\Theta_{ij}dxdt .
%\end{align*}
To estimate the first term of the R.H.S., we observe 
\begin{align*}
\rho(\Theta_{jj}-\Theta_{ii}) &= \int_{\mathbb{R}^3}F\left(v_j^2-v_i^2\right)dv +\rho U_j^2-\rho U_i^2.
\end{align*}
Therefore, substituting $F=m+\sqrt{m}f$ and applying the H\"{o}lder inequality, we get 
\begin{align}\label{jj-ii}
\begin{split}
\big|\rho(\Theta_{jj}-\Theta_{ii})\big| &\leq  \int_{\mathbb{R}^3}(m+\sqrt{m}f)\left(v_j^2-v_i^2\right)dv +\frac{(\rho U_j)^2+(\rho U_i)^2}{\rho} \cr
&\leq C \| f \|_{L^2_v} +C \| f \|_{L^2_v}^2,
\end{split}
\end{align}
where we used the lower bound of $\rho$ in Lemma \ref{macro} (1) and 
\begin{align*}
\int_{\mathbb{R}^3}m\left(v_j^2-v_i^2\right)dv = 0 .
\end{align*}
We insert \eqref{jj-ii} into \eqref{d-d0}, and apply $\rho U/\rho \leq C\| f \|_{L^2_v}$ and $\rho \Theta \leq C\| f \|_{L^2_v}$ to obtain 
\begin{align*}
\bigg|\int_{\mathbb{T}^3}d_i(x,t)dx\bigg| &\leq C \int_0^t \int_{\mathbb{T}^3} \| f \|_{L^2_v}^2 + \| f \|_{L^2_v}^3dxdt .
\end{align*}
We then apply the Sobolev embedding $H^2 \subset\subset L^{\infty}$ to bound
\begin{align*}
\int_{\mathbb{T}^3} \| f \|_{L^2_v}^3dx \leq \sup_{x\in \mathbb{T}^3}\| f \|_{L^2_v}  \| f \|_{L^2_{x,v}}^2 \leq \sum_{|\alpha| \leq 2}\| \partial^{\alpha}f \|_{L^2_{x,v}}  \| f \|_{L^2_{x,v}}^2 \leq \sqrt{\mathcal{E}(t)}\| f \|_{L^2_{x,v}}^2.
\end{align*}
Therefore,  
\begin{align}\label{it}
\bigg|\int_{\mathbb{T}^3}d_i(x,t)dx\bigg| &\leq (1+\sqrt{M_0}) \int_0^t  \| f \|_{L^2_{x,v}}^2dt \leq  (1+\sqrt{M_0})\mathcal{E}(t),
\end{align}
where $M_0$ is from Theorem \ref{local} (1). We note that this is why we use the energy functional with the time integration of the production
term.
Combining \eqref{it} with \eqref{d1} gives 
\begin{align*}
\bigg|\int_{\mathbb{T}^3}\partial^{\gamma} d_i(x,t)dx\bigg| &\leq C(1+\sqrt{M_0}) \mathcal{E}(t),
\end{align*}
for $|\gamma|= 0,1 $. Substituting it into \eqref{poin d} yields 
\begin{align*}
	\|\partial^{\gamma} d_i\|_{L^2_x} &\leq \|\nabla_x \partial^{\gamma} d_i\|_{L^2_x} +C(1+\sqrt{M_0}) \mathcal{E}(t) \cr
	&\leq C  \left(\|\partial^{\gamma}\tilde{l}\|_{L^2_x}+\|\partial^{\gamma}\tilde{h}\|_{L^2_x}\right) +C(1+\sqrt{M_0})\mathcal{E}(t),
\end{align*}
where we used the result of (Step 2).\newline

\noindent (3) The case of $2 \leq 
|\gamma|\leq N$: We have from $\eqref{system}_5$
\begin{align*}
\partial_t^{|\gamma|} d_i &=\partial_t^{|\gamma|-1}l_{is} + \partial_t^{|\gamma|-1}h_{is} - \partial_t^{|\gamma|-1}\partial_{x_i}c.
\end{align*}
Since the last term $\partial_t^{|\gamma|-1}\partial_{x_i}c$ has at least one time derivative, we can apply the estimate of $\partial_t c$ in (Step 1):
\begin{align*}
\|\partial_t^{|\gamma|} d_i\|_{L^2_x} &\leq  C\left(\|\partial_t^{|\gamma|-1}l_{is}\|_{L^2_x}+\|\partial_t^{|\gamma|-1}h_{is}\|_{L^2_x}+\| \partial_t^{|\gamma|-1}\partial_{x_i}c\|_{L^2_x}\right) \cr
&\leq C \sum_{|\alpha|\leq N-1}(\|\partial^{\alpha}\tilde{l}\|_{L_{x}^2} + \|\partial^{\alpha}\tilde{h}\|_{L_{x}^2}).
\end{align*}
Finally, we combine (1)-(3) to get the desired result: 
\begin{align*}
\|\partial^{\gamma} d_i\|_{L^2_x} &\leq  C \sum_{|\alpha|\leq N-1}(\|\partial^{\alpha}\tilde{l}\|_{L_{x}^2} + \|\partial^{\alpha}\tilde{h}\|_{L_{x}^2}) +C \mathcal{E}(t).
\end{align*}
\newline

%%%%%%%%%%%%%%%%%%%%%%%%%%%%%%%%%%%%%%%%%%%%%%%%%%%%%%%%%%%%%%%%%%%%%%%%%%%%%%%%%%%%%%%%%%%%%%%%%%%%%%%%%%%%%%%%%%%%%%%%%%%%%%%%%%%%%%%%%%%
\noindent$\bullet$ {\bf Step 4:} We first consider the estimate of $c$ which has at least one spatial derivative. We take $\partial^{\alpha}$ on $\eqref{system}_5$ to get
\begin{align*}
	\partial^{\alpha}\partial_{x_i}c &= \partial^{\alpha}l_{is} + \partial^{\alpha}h_{is}-\partial^{\alpha}\partial_td_i.
\end{align*}
Using (Step 3), we have
\begin{align*}
	\|\partial^{\alpha}\partial_{x_i}c\|_{L^2_x} &\leq  \|\partial^{\alpha}l_{is}\|_{L^2_x} + \|\partial^{\alpha}h_{is}\|_{L^2_x}+\|\partial^{\alpha}\partial_td_i \|_{L^2_x} \cr
	&\leq C \sum_{|\alpha|\leq N-1}(\|\partial^{\alpha}\tilde{l}\|_{L_{x}^2} + \|\partial^{\alpha}\tilde{h}\|_{L_{x}^2}) +C \mathcal{E}(t).
\end{align*}
Finally, we use the Poincar\'{e} inequality to get the estimate of $c$ without derivative:
\begin{align*}
\|c\|_{L^2_x} &\leq\|\nabla_xc\|_{L^2_x}\leq C \sum_{|\alpha|\leq N-1}(\|\partial^{\alpha}\tilde{l}\|_{L_{x}^2} + \|\partial^{\alpha}\tilde{h}\|_{L_{x}^2}) +C \mathcal{E}(t).
\end{align*}
This completes the proof of Proposition \ref{abcd}.
\end{proof}
Finally, we need to estimate the R.H.S. of \eqref{split}.
\begin{lemma} \label{52} Suppose that $\mathcal{E}(t)$ is sufficiently small.
%\begin{align*}
%\sum_{|\alpha|\leq N}\|\partial^{\alpha}f\|_{L^2_{x,v}}^2 \leq M_0,
%\end{align*}	
Then we have
\begin{align*}
&(1) \ \sum_{|\alpha|\leq N-1} \|\partial^{\alpha}\tilde{l}\|_{L^2_x} \leq C \sum_{|\alpha|\leq N}\|(I-P_{Pr})\partial^{\alpha}f\|_{L^2_{x,v}}, \cr 
&(2) \ \sum_{|\alpha|\leq N} \|\partial^{\alpha}\tilde{h}\|_{L^2_x} \leq CM_0 \sum_{|\alpha|\leq N}\|\partial^{\alpha}f\|_{L^2_{x,v}}.
\end{align*}
\end{lemma}
\begin{proof}
This estimates are standard (See for example \cite{MR1908664,MR2000470,MR2779616}). The only difference is that, as we can see in \eqref{basis}, the number of basis changes from $13$ to $19$, so we omit it.
\end{proof}

Now we are ready to prove Theorem \ref{full coer}.
\subsection{Proof of Theorem \ref{full coer}}
From Proposition \ref{abcd}, we have Lemma \ref{52},
\begin{align*}
\sum_{|\alpha|\leq N}\|P_{Pr}\partial^{\alpha}f\|^2_{L^2_{x,v}} &\leq C\sum_{|\alpha|\leq N} \left(\|\partial^{\alpha}a\|^2_{L^2_x}+\|\partial^{\alpha}b\|^2_{L^2_x}+\|\partial^{\alpha}c\|^2_{L^2_x}+\|\partial^{\alpha}d\|^2_{L^2_x}\right) \cr
&\leq C \sum_{|\alpha|\leq N-1}\left( \|\partial^{\alpha}\tilde{l}\|_{L_{x}^2} + \|\partial^{\alpha}\tilde{h}\|_{L_{x}^2}+ C\mathcal{E}(t)\right)^2 \cr
&\leq  C \sum_{|\alpha|\leq N}\|(I-P_{Pr})\partial^{\alpha}f\|^2_{L^2_{x,v}} +CM_0 \sum_{|\alpha|\leq N}\|\partial^{\alpha}f\|^2_{L^2_{x,v}}+ C\mathcal{E}^2(t).
\end{align*}
Adding $\sum\|(I-P_{Pr})\partial^{\alpha}f\|^2_{L^2_{x,v}} $ on both side, we get
\begin{align*}
\sum_{|\alpha|\leq N}\|\partial^{\alpha}f\|^2_{L^2_{x,v}} &\leq C \sum_{|\alpha|\leq N}\|(I-P_{Pr})\partial^{\alpha}f\|^2_{L^2_{x,v}} + C\mathcal{E}^2(t),
\end{align*}
for the sufficiently small $ M_0$. This, combined with the degenerate coercivity estimate in Proposition \ref{dichotomy} (2), lead to the following modified coercivity estimate for sufficiently small $\mathcal{E}(t)$: 
\begin{align*}
\sum_{|\alpha|\leq N}\langle L_{Pr}\partial^{\alpha}f, \partial^{\alpha}f\rangle_{L^2_{x,v}} \leq -\delta \sum_{|\alpha|\leq N}\|\partial^{\alpha}f\|_{L^2_{x,v}}^2+C\mathcal{E}^2(t),
\end{align*}
for some positive constant $\delta>0$. This completes the proof of Theorem \ref{full coer}.

\section{Global existence}
In this section, we prove Theorem \ref{theorem}. We will only prove
\begin{align}\label{want}
	\sum_{\substack{|\alpha|+|\beta|\leq N \cr |\beta|\leq m}}\left\{C_{m_1}\frac{d}{dt}\|\partial^{\alpha}_{\beta}f\|_{L^2_{x,v}}^2 +\delta_m \sum_{|\alpha|\leq N}\|\partial^{\alpha}_{\beta}f\|_{L^2_{x,v}}^2\right\} &\leq C_{m_2} \mathcal{E}^2(t),
\end{align}
for some positive constants $C_{m_1}$, $C_{m_2}$, and $\delta_m$ for $0 \leq m \leq N $, since the standard argument in \cite{MR1908664,MR2000470,MR2779616} leads to  the desired result.  \newline
{\bf Proof of \eqref{want}:} Let $f$ be a local-in-time solution constructed in Theorem \ref{local}. We apply the induction argument for the momentum derivative $|\beta|=m$. For $m=0$, taking $\partial^{\alpha}$ on \eqref{perturb} and applying inner product with $\partial^{\alpha}f$ give
\begin{align*}
\frac{1}{2}\frac{d}{dt}\|\partial^{\alpha}f\|_{L^2_{x,v}}^2 = \frac{1}{\tau_0}\langle \partial^{\alpha}f, L_{Pr}\partial^{\alpha}f\rangle_{L^2_{x,v}} + \langle \partial^{\alpha}f, \partial^{\alpha}\Gamma(f)\rangle_{L^2_{x,v}}.
\end{align*}
We then apply  the coercivity estimate \eqref{full coer0} in the case $Pr>0$,  and apply Theorem \ref{full coer} in the degenerate case $Pr=0$, to obtain  
\begin{align*}
\frac{1}{2}\frac{d}{dt}\|\partial^{\alpha}f\|_{L^2_{x,v}}^2 +\delta \|\partial^{\alpha}f\|_{L^2_{x,v}}^2&\leq \langle \partial^{\alpha}f, \partial^{\alpha}\Gamma(f)\rangle_{L^2_{x,v}}.
\end{align*}
For the nonlinear term, we apply Proposition \ref{prop} to have
\begin{align*}
\langle \partial^{\alpha}f, \partial^{\alpha}\Gamma(f)\rangle_{L^2_{x,v}}	 &\leq  C\sqrt{\mathcal{E}(t)}\sum_{|\alpha_1|+|\alpha_2|\leq |\alpha|}\int_{\mathbb{T}^3}\|\partial^{\alpha_1}f\|_{L^2_{v}}\|\partial^{\alpha_2}f\|_{L^2_{v}}\|\partial^{\alpha}f\|_{L^2_{v}}dx.
\end{align*}
Without loss of generality, we assume that $|\alpha_1| \leq |\alpha_2|$, 
%\begin{align*}
%\langle \partial^{\alpha}f, \partial^{\alpha}\Gamma(f)\rangle_{L^2_{x,v}}	 &\leq  C\sqrt{\mathcal{E}(t)}\sum_{|\alpha_1|\leq |\alpha|/2}\sup_{x\in\mathbb{T}^3}\|\partial^{\alpha_1}f\|_{L^2_{v}}\sum_{|\alpha_2|\leq |\alpha|}\int_{\mathbb{T}^3}\|\partial^{\alpha_2}f\|_{L^2_{v}}\|\partial^{\alpha}f\|_{L^2_{v}}dx.
%\end{align*}
and employ the Sobolev embedding $H^2 \subset\subset L^{\infty}$ to get 
\begin{align*}
\langle \partial^{\alpha}f, \partial^{\alpha}\Gamma(f)\rangle_{L^2_{x,v}}	 &\leq  C\sqrt{\mathcal{E}(t)}\bigg(\sum_{|\alpha_1|\leq |\alpha|}\|\partial^{\alpha_1}f\|_{L^2_{x,v}}\|\bigg)^2 \|\partial^{\alpha}f\|_{L^2_{x,v}} \leq \mathcal{E}^2(t).
\end{align*}
Thus we obtain the following estimate of $|\beta|=0$:
\begin{align*}
\mathcal{E}^{\alpha}_0 ~:~ \frac{1}{2}\frac{d}{dt}\|\partial^{\alpha}f\|_{L^2_{x,v}}^2 +\delta \|\partial^{\alpha}f\|_{L^2_{x,v}}^2&\leq C\mathcal{E}^2(t).
\end{align*}
Now we consider the case $|\beta|=m>0$. We take $\partial^{\alpha}_{\beta}$ on \eqref{perturb}
%\begin{align*}
%	\partial_t\partial^{\alpha}_{\beta} f+v\cdot \nabla_x \partial^{\alpha}_{\beta} f +\frac{1}{\tau}\partial^{\alpha}_{\beta} f+ \sum_{i=1}^3 \partial^{\alpha+e_i}_{\beta-e_i}\partial^{\alpha}_{\beta} f&= \frac{1}{\tau}(\partial_{\beta}P_{Pr}(\partial^{\alpha}f)+\partial^{\alpha}_{\beta}\Gamma(f)),
%\end{align*}
and apply inner product with $\partial^{\alpha}_{\beta}f$.
\begin{align}\label{Eab}
\begin{split}
\mathcal{E}^{\alpha}_{\beta} ~:~ &\frac{1}{2}\frac{d}{dt}\|\partial^{\alpha}_{\beta}f\|_{L^2_{x,v}}^2 +\frac{1}{\tau_0}\|\partial^{\alpha}_{\beta}f\|_{L^2_{x,v}}^2 \cr
&\leq \sum_{i=1}^3 \langle \partial^{\alpha}_{\beta}f,  \partial^{\alpha+\bar{k}_i}_{\beta-k_i}\partial^{\alpha}_{\beta} f \rangle_{L^2_{x,v}} + \frac{1}{\tau_0}\langle \partial^{\alpha}_{\beta}f, \partial_{\beta}P_{Pr}(\partial^{\alpha}f) \rangle_{L^2_{x,v}}+ \langle \partial^{\alpha}_{\beta}f,\partial^{\alpha}_{\beta}\Gamma(f) \rangle_{L^2_{x,v}}.
\end{split}
\end{align}
The first two terms on the second line can be estimated by Young's inequality:
\begin{align*}
\langle \partial^{\alpha}_{\beta}f,  \partial^{\alpha+\bar{k}_i}_{\beta-k_i}\partial^{\alpha}_{\beta} f \rangle_{L^2_{x,v}} \leq \frac{\epsilon}{2}\|\partial^{\alpha}_{\beta}f\|_{L^2_{x,v}}^2+\frac{1}{2\epsilon} \| \partial^{\alpha+\bar{k}_i}_{\beta-k_i}f\|_{L^2_{x,v}}^2,
\end{align*}
and
\begin{align*}
\langle \partial^{\alpha}_{\beta}f, \partial_{\beta}P_{Pr}(\partial^{\alpha}f) \rangle_{L^2_{x,v}} \leq \frac{\epsilon}{2}\|\partial^{\alpha}_{\beta}f\|_{L^2_{x,v}}^2+\frac{1}{2\epsilon}\| \partial^{\alpha}f\|_{L^2_{x,v}}^2,
\end{align*}
where we used 
\begin{align*}
\| \partial_{\beta}P_{Pr}(\partial^{\alpha}f)\|_{L^2_{x,v}}^2 \leq \| \partial^{\alpha}f\|_{L^2_{x,v}}^2.
\end{align*}
To estimate the last term of \eqref{Eab}, we apply Proposition \ref{prop}.
\begin{multline*}
\mathcal{E}^{\alpha}_{\beta} ~:~ \frac{1}{2}\frac{d}{dt}\|\partial^{\alpha}_{\beta}f\|_{L^2_{x,v}}^2 +\frac{1}{\tau_0}\|\partial^{\alpha}_{\beta}f\|_{L^2_{x,v}}^2 \cr
\leq C_{\epsilon}\|\partial^{\alpha}_{\beta}f\|_{L^2_{x,v}}^2+\frac{1}{2\epsilon}\sum_{i=1}^3  \| \partial^{\alpha+\bar{k}_i}_{\beta-k_i}f\|_{L^2_{x,v}}^2 + \frac{1}{2\epsilon}\|\partial^{\alpha}f\|_{L^2_{x,v}}^2+ C\mathcal{E}^2(t).
\end{multline*}
For sufficiently small $\epsilon$, the right-hand side of $\|\partial^{\alpha}_{\beta}f\|_{L^2_{x,v}}^2$ can be absorbed in the left-hand side of that. Once we take $\displaystyle \sum_{|\alpha|+|\beta|\leq N}$ and $\displaystyle\sum_{|\beta|=m+1}$ on each side, then the 2-nd and the 3-rd terms of the second line is bounded by the induction hypothesis: 
\begin{align*}
\sum_{\substack{|\alpha|+|\beta|\leq N \cr |\beta|=m+1}} \left(\sum_{i=1}^3  \| \partial^{\alpha+\bar{k}_i}_{\beta-k_i}f\|_{L^2_{x,v}}^2 + \|\partial^{\alpha}f\|_{L^2_{x,v}}^2 \right) \leq C_{m}\sum_{\substack{|\alpha|+|\beta|\leq N \cr |\beta|\leq m}}\mathcal{E}^{\alpha}_{\beta} + C_{0}\sum_{|\alpha|\leq N }\mathcal{E}^{\alpha} \leq C\mathcal{E}^2(t).
\end{align*}
Thus we have the desired result. This completes the proof.\\

\noindent {\bf Acknowledgement:}
G.-C. Bae is supported by the National Research Foundation of Korea(NRF) grant funded by the Korea government(MSIT) (No. 2021R1C1C2094843). S.-B. Yun is supported by Samsung Science and Technology Foundation under Project Number SSTF-BA1801-02.

\section{Appendix}
In this part, we prove the conservation laws \eqref{conserv}, the $H$-theorem \eqref{Hthm}, and the 
cancellation property \eqref{rq} of the Shakhov model. We present the proof in detail for the reader's convenience.
\begin{lemma} The Shakhov model satisfies the conservation laws \eqref{conserv} and the additional cancellation property \eqref{rq}.
\end{lemma}
\begin{proof}
It is enough to show that 
\begin{align*}
&(1) ~\int_{\mathbb{R}^3}\mathcal{S}_{Pr}(F)(x,v,t)dv = \rho, \cr
&(2) ~\int_{\mathbb{R}^3}(v-U)\mathcal{S}_{Pr}(F)(x,v,t)dv = 0,\cr
&(3) ~\int_{\mathbb{R}^3}|v-U|^2\mathcal{S}_{Pr}(F)(x,v,t)dv = 3\rho T, \\
&(4) ~ \int_{\mathbb{R}^3}(\mathcal{S}_{Pr}(F)-F)(v_i-U_i)|v-U|^2dv = -Pr q_i.
\end{align*}
(1) We integrate the Shakhov operator \eqref{Shakhov operator} with respect to $dv$ and take the change of variable $(v-U) \rightarrow v$ to get 
\begin{align*}
\int_{\mathbb{R}^3}\mathcal{S}_{Pr}(F) dv = \rho+ \int_{\mathbb{R}^3}\frac{\rho}{\sqrt{2\pi T}^3}\exp\left(-\frac{|v|^2}{2T}\right)\frac{1-Pr}{5}\frac{q\cdot v}{\rho T^2}\left(\frac{|v|^2}{2T}-\frac{5}{2}\right)dv = \rho.
\end{align*}
(2) Multiplying  $\mathcal{S}_{Pr}(F)$ by $(v_i-U_i)$ and applying the change of variable $(v-U) \rightarrow v$ yield
\begin{align*}
\int_{\mathbb{R}^3}(v_i-U_i)\mathcal{S}_{Pr}(F) dv &=  \int_{\mathbb{R}^3}v_i\frac{\rho}{\sqrt{2\pi T}^3}\exp\left(-\frac{|v|^2}{2T}\right)\frac{1-Pr}{5}\frac{q\cdot v}{\rho T^2}\left(\frac{|v|^2}{2T}-\frac{5}{2}\right)dv \cr
%&=\frac{1-Pr}{5}\frac{\rho}{\sqrt{2\pi T}^3}\frac{1}{\rho T^2} \int_{\mathbb{R}^3}v_i(q_1v_1+q_2v_2+q_3v_3)\left(\frac{|v|^2}{2T}-\frac{5}{2}\right)\exp\left(-\frac{|v|^2}{2T}\right)dv \cr
&=\frac{1-Pr}{5}\frac{\rho}{\sqrt{2\pi T}^3}\frac{1}{\rho T^2} \int_{\mathbb{R}^3}q_iv_i^2\left(\frac{|v|^2}{2T}-\frac{5}{2}\right)\exp\left(-\frac{|v|^2}{2T}\right)dv \cr
&=\frac{1-Pr}{5}\frac{\rho}{\sqrt{2\pi T}^3}\frac{1}{\rho T^2} \frac{1}{3}\int_{\mathbb{R}^3}q_i|v|^2\left(\frac{|v|^2}{2T}-\frac{5}{2}\right)\exp\left(-\frac{|v|^2}{2T}\right)dv.
\end{align*}
We then take another change of variable $v/\sqrt{2T}\rightarrow v$ to obtain
\begin{align*}
\int_{\mathbb{R}^3}(v_i-U_i)\mathcal{S}_{Pr}(F) dv&=\frac{1-Pr}{15}\frac{q_i}{4\pi^{3/2}T^4}\left(\int_{\mathbb{R}^3}|v|^4e^{-|v|^2}dv-\frac{5}{2}\int_{\mathbb{R}^3}|v|^2e^{-|v|^2}dv\right)=0,
\end{align*}
where we used
\begin{align*}
\int_{\mathbb{R}^3}|v|^4e^{-|v|^2}dv= \frac{15\pi^{\frac{3}{2}}}{4}, \qquad 
\int_{\mathbb{R}^3}|v|^2e^{-|v|^2}dv=\frac{3\pi^{\frac{3}{2}}}{2}.
\end{align*}
(3) We integrate  $\mathcal{S}_{Pr}(F)$ with respect to $|v-U|^2dv$ and the change of variable $(v-U) \rightarrow v$: 
\begin{align*}
\int_{\mathbb{R}^3}|v-U|^2\mathcal{S}_{Pr}(F) dv &= 3\rho T + \int_{\mathbb{R}^3}|v|^2\frac{\rho}{\sqrt{2\pi T}^3}\exp\left(-\frac{|v|^2}{2T}\right)\frac{1-Pr}{5}\frac{q\cdot v}{\rho T^2}\left(\frac{|v|^2}{2T}-\frac{5}{2}\right)dv \cr
&=3\rho T.
\end{align*}
(4) Integrating the Shakhov operator $\mathcal{S}_{Pr}(F)$ with respect to $(v_i-U_i)|v-U|^2 dv$ gives
\begin{align*}
\int_{\mathbb{R}^3}  \mathcal{S}_{Pr}(F)(v_i-U_i)|v-U|^2 dv
&=\int_{\mathbb{R}^3} v_i|v|^2\left[1+\frac{1-Pr}{5}\frac{q\cdot v}{
	\rho T^2}\left(\frac{|v|^2}{2T}-\frac{5}{2}\right)\right]\frac{\rho}{\sqrt{2\pi T}^3}e^{-\frac{|v|^2}{2T}}  dv \cr
&=\frac{1-Pr}{5}\frac{q_i}{\rho T^2}\frac{\rho}{\sqrt{2\pi T}^3}\int_{\mathbb{R}^3} \left[\left(\frac{v_i^2|v|^4}{2T}-\frac{5}{2}v_i^2|v|^2\right)\right]e^{-\frac{|v|^2}{2T}}  dv \cr
&=\frac{1-Pr}{15}\frac{q_i}{\rho T^2}\frac{\rho}{\sqrt{2\pi T}^3}\int_{\mathbb{R}^3} \left[\left(\frac{|v|^6}{2T}-\frac{5}{2}|v|^4\right)\right]e^{-\frac{|v|^2}{2T}}  dv \cr
%&=\int_{\mathbb{T}^3} \frac{1-Pr}{5}\frac{q_i}{\rho T^2}\frac{\rho}{\sqrt{2\pi T}^3}4\pi^{3/2}\frac{5}{16}(2T)^{7/2}  dx \cr
&=(1-Pr)q_i.
\end{align*}
Then by the definition of the heat flux $q_i$, we have 
\begin{align*}
\int_{\mathbb{R}^3}(\mathcal{S}_{Pr}(F)-F)(v_i-U_i)|v-U|^2dv = -Pr q_i.
\end{align*}
\end{proof}

\begin{lemma}\emph{\cite{shakhov1968generalization}} The Shakhov model satisfies the $H$-theorem \eqref{Hthm} when the distribution function $F(x,v,t)$ is sufficiently close to the global Maxwellian in the sense that
\begin{align*}
|\rho-1|+|U|+|T-1|+|q| \ll 1.
\end{align*}
\begin{remark}
We remark that the above smallness condition is satisfied by the solution derived in Theorem \ref{theorem}.
\end{remark}
\end{lemma}
\begin{proof}
We take $(1+\ln F) dvdx$ on both sides of \eqref{Shakhov}:
\begin{align*}
\frac{d}{dt}\int_{\mathbb{T}^3\times \mathbb{R}^3} F\ln F dvdx &= \int_{\mathbb{T}^3\times \mathbb{R}^3}(\mathcal{S}_{Pr}(F)-F)\ln F dvdx \cr
&= \int_{\mathbb{T}^3\times \mathbb{R}^3}(\mathcal{S}_{Pr}(F)-F)\ln \frac{F}{\mathcal{S}_{Pr}(F)} dvdx  \cr
&\quad + \int_{\mathbb{T}^3\times \mathbb{R}^3}(\mathcal{S}_{Pr}(F)-F)\ln \mathcal{S}_{Pr}(F) dvdx.
\end{align*}
Since the first term is non-positive, we only consider the second term.
We expand $\ln \mathcal{S}_{Pr}(F)$ with respect to $q$ as
\begin{align*}
\ln \mathcal{S}_{Pr}(F) = \ln \mathcal{M} + \frac{1-Pr}{5}\frac{(v-U)}{\rho T^2}\left(\frac{|v-U|^2}{2T}-\frac{5}{2}\right)\cdot q + O(q^2),
\end{align*}
so that
\begin{align*}
\int_{\mathbb{T}^3\times \mathbb{R}^3}&(\mathcal{S}_{Pr}(F)-F)\ln \mathcal{S}_{Pr}(F) dvdx \cr
&= \int_{\mathbb{T}^3\times \mathbb{R}^3}(\mathcal{S}_{Pr}(F)-F)\times \left[\ln \mathcal{M} + \frac{1-Pr}{5}\frac{q\cdot(v-U) }{\rho T^2}\left(\frac{|v-U|^2}{2T}-\frac{5}{2}\right) + O(q^2)\right]dvdx \cr
&= \int_{\mathbb{T}^3}\frac{1-Pr}{5\rho T^2}\int_{\mathbb{R}^3}(\mathcal{S}_{Pr}(F)-F)\times \left[q\cdot (v-U)\left(\frac{|v-U|^2}{2T}-\frac{5}{2}\right)\right]dvdx +O(q^2)\cr
&= \int_{\mathbb{T}^3} \frac{-Pr(1-Pr)|q|^2}{10\rho T^3}dx +O(q^2).
\end{align*}
In the last line, we used
\begin{align*}
\int_{\mathbb{R}^3}(\mathcal{S}_{Pr}(F)-F)(v-U)|v-U|^2dv = -Pr q(x,t).
\end{align*}
Since $\rho$ and $T$ have lower bounds by the assumption, for sufficiently small $q$, we have the desired result.
\end{proof}

\end{document}